\documentclass[10pt]{article}
\usepackage{graphicx,epsfig,amsmath,amsthm,amssymb}

\def\argmax{\operatornamewithlimits{arg\,max}}
\newcommand{\argmin}{\mathop{\mathrm{arg\min}}}

\newtheorem{definition}{Definition}
\newtheorem{lemma}{Lemma}
\newtheorem{clm}{Claim}
\newtheorem{proposition}{Proposition}
\newtheorem{assumption}{Assumption}

\newtheorem{corollary}{Corollary}

\newcommand{\beqano}{\begin{eqnarray*}}
\newcommand{\eeqano}{\end{eqnarray*}}
\newcommand{\beqa}{\begin{eqnarray}}
\newcommand{\eeqa}{\end{eqnarray}}
\newcommand{\dsp}{\displaystyle}
\newcommand{\ra}{\rightarrow}
\newcommand{\bi}{\begin{itemize}}
\newcommand{\ei}{\end{itemize}}

\newcommand{\be}{\begin{enumerate}}
\newcommand{\ee}{\end{enumerate}}

\newcommand{\cB}{\mathcal{B}}
\newcommand{\cL}{\mathcal{L}}
\newcommand{\cS}{\mathcal{S}}
\newcommand{\cR}{\mathcal{R}}
\newcommand{\cF}{\mathcal{F}}
\newcommand{\cH}{\mathcal{H}}

\newcommand{\cX}{\mathcal{X}}

\newcommand{\vQ}{{\bf Q}}
\newcommand{\vS}{{\bf S}}
\newcommand{\vA}{{\bf A}}
\newcommand{\vU}{{\bf U}}
\newcommand{\vx}{{\bf x}} 
\newcommand{\vy}{{\bf y}} 
\newcommand{\vr}{{\bf r}} 
\newcommand{\vc}{{\bf c}} 
\newcommand{\vZero}{{\bf 0}} 
\newcommand{\vOne}{{\bf 1}} 
\newcommand{\vR}{{\bf R}}

\newcommand{\vQtil}{\widetilde{\bf Q}}

\newcommand{\vUtil}{\widetilde{\bf U}}
\newcommand{\Util}{\widetilde{U}}
\newcommand{\vctil}{\widetilde{\bf c}}

\newcommand{\vlam}{{\mbox{\boldmath{$\lambda$}}}}
\newcommand{\vsig}{{\mbox{\boldmath{$\sigma$}}}}

\newcommand{\vmu}{{\mbox{\boldmath{$\mu$}}}}
\newcommand{\vnu}{{\mbox{\boldmath{$\nu$}}}}
\newcommand{\veps}{\epsilon}

\newcommand{\bQ}{\bar{Q}}
\newcommand{\bZ}{\bar{Z}}
\newcommand{\bB}{\overline{B}}

\newcommand{\bvQ}{{\overline{\vQ}}}
\newcommand{\bPhi}{{\overline{\Phi}}}

\newcommand{\bE}{\mathbb{E}}
\newcommand{\bP}{\mathbb{P}}
\newcommand{\bR}{\mathbb{R}}

\newcommand{\one}{\mathcal{I}}
\newcommand{\e}{^{(\epsilon)}}
\newcommand{\euk}{\epsilon}

\newcommand{\uk}{^{(k)}}

\newcommand{\uve}{^{(\veps)}}
\newcommand{\uvek}{^{(\veps,k)}}
\newcommand{\La}{\langle \>}
\newcommand{\Ra}{\> \rangle}

\newcommand{\cJ}{\mathcal{J}}
\newcommand{\uj}{^{(j)}}
\newcommand{\ujk}{^{(j,k)}}
\newcommand{\bK}{{K}}
\newcommand{\bcR}{{\mathcal{R}}}
\newcommand{\bcF}{{\mathcal{F}}}
\newcommand{\bcH}{{\mathcal{H}}}

\newcommand{\bvc}{{\bf c}}
\newcommand{\bb}{{b}}
\newcommand{\vs}{{\bf s}}
\newcommand{\vpsi}{{\mbox{\boldmath{$\psi$}}}}

\title{Asymptotically Tight Steady-State Queue Length Bounds Implied By Drift Conditions\footnote{This is an updated version of the paper that appeared in QUESTA in 2012. We have corrected the statements of some of propositions and simplified some of the notation.}}
\author{Atilla Eryilmaz\\Department of ECE\\The Ohio State University\\eryilmaz@osu.edu \and R. Srikant\\
Department of ECE and CSL\\University of Illinois\\rsrikant@illinois.edu}

\begin{document}

\maketitle

\begin{abstract}
The Foster-Lyapunov theorem and its variants serve as the primary tools for studying the stability of queueing systems. In addition, it is well known that setting the drift of the Lyapunov function equal to zero in steady-state provides bounds on the expected queue lengths. However, such bounds are often very loose due to the fact that they fail to capture resource pooling effects. The main contribution of this paper is to show that the approach of ``setting the drift of a Lyapunov function equal to zero" can be used to obtain bounds on the steady-state queue lengths which are tight in the heavy-traffic limit. The key is to establish an appropriate notion of state-space collapse in terms of steady-state moments of weighted queue length differences, and use this state-space collapse result when setting the Lyapunov drift equal to zero. As an application of the methodology, we prove the steady-state equivalent of the heavy-traffic optimality result of Stolyar for wireless networks operating under the MaxWeight scheduling policy.
\end{abstract}

\section{Introduction}
\label{sec:intro}

The performance of control policies in queueing systems is evaluated by studying the sum of appropriately weighted queue lengths, either in steady-state and along almost every sample path. However, deriving optimal control policies is difficult because any stochastic optimal control formulation of the problem is often intractable. An alternative is to study the system in heavy-traffic, i.e., when the vector of exogenous arrival rates to the queueing system is close to the capacity region in the network. In such regimes, the behavior of the network often simplifies: a multi-dimensional state description of the queueing system reduces to a single dimension (or to a small number of dimensions) and it is easier to reason about the optimality of the control policies in one dimension. This behavior of the queueing system in heavy-traffic, called \emph{state-space collapse}, is at the heart of most heavy-traffic optimality results.

Heavy-traffic analysis of queueing systems using diffusion limits was initiated by Kingman in \cite{kin62b} and state-space collapse was observed for priority queues by Whitt in \cite{whi71}. The use of state-space collapse to study heavy-traffic optimality was introduced by Foschini and Salz in \cite{fossal78} in their classic paper on join-the-shortest queue (JSQ) routing. Since then, the methodology and applicability of this technique have been extended in a number of papers; see, for example, the works of Reiman \cite{rei83}, Bramson \cite{bra98}, Williams \cite{wil98}, Harrison \cite{har98}, Harrison and Lopez \cite{harlop99}, and Bell and Williams \cite{belwil05}. This list of papers is by no means exhaustive, it is only meant to be a representative sampling of the papers in the area. Many of these papers which consider multi-queue models served by multiple resources rely on the so-called \emph{resource pooling condition}, under which the behavior of the queueing system under study in heavy-traffic is governed by a single bottleneck resource. This results in the state-space collapse mentioned earlier, which is critical to establishing heavy-traffic optimality. In addition, these papers also implicitly assume that the scheduling policy in the queueing system is work conserving, i.e., backlogged work is served at the maximum possible rate by each station. In a seminal paper on generalized switches, Stolyar extended the notion of state-space collapse and resource pooling to systems where such per-node work-conserving policies are hard to define \cite{sto04}. In particular, he showed that a class of scheduling policies called the MaxWeight policies, introduced by Tassiulas and Ephremides \cite{taseph92} (see \cite{leeagr01,neemodroh05,erysriper05,mey07a} for extensions) is heavy-traffic optimal in an appropriate sense. While Stolyar's work considered single-hop traffic only, the proof of heavy-traffic optimality in the multi-hop case was provided by Dai and Lin \cite{dailin08}. Extensions to other types of scheduling policies were presented in \cite{shasristo04} by Shakkottai, Stolyar and Srikant. It should be noted that the MaxWeight policy was shown to be optimal at all traffic loads for a simple wireless network model (with symmetric Bernoulli arrival and service processes) by Tassiulas and Ephremides \cite{taseph93}, who also obtained optimal policies (which are not of MaxWeight type) for wireless networks where the links are arranged in a line \cite{taseph94}.

In general, state-space collapse does not lead to a one-dimensional state space. When the state-space collapse is to a multi-dimensional state, it is often harder to prove optimality in heavy-traffic. A model of the Internet with multi-dimensional state-space collapse has been considered by  Kang, Kelly, Lee and Williams \cite{kankelleewil09}. While the resource allocation policy considered there is not optimal in heavy-traffic, an important contribution there is to show that the expected workload is only a function of the number of resources in the system and not the number of flows in the system. State-space collapse is key to establishing such a result. Multi-dimensional state-space collapse has also been studied by Shah and Wischik for generalized switches \cite{shawis07}. In addition to recovering many of the earlier results for other models as special cases, a key contribution in \cite{shawis07} is the introduction and study of appropriate notions of optimality when the arrival rates lie outside the capacity region of the system. Multi-dimensional state-space collapse for a very simple four-link wireless network has been considered by Ji, Athanasopoulou and Srikant in \cite{jiathsri10} who derive the heavy-traffic optimal policy for a network which has two bottleneck resources in the heavy-traffic limit. In contrast to heavy-traffic limits, Venkatramanan and Lin have shown the optimality of the MaxWeight policies in a large-deviations sense \cite{venlin09}.

Much of work on heavy-traffic analysis of queueing systems relies on showing that a scaled version of the queue lengths in the system converges to a regulated Brownian motion. The typical result shows sample-path optimality in scaled time over a finite time interval. Often, the results allow a straightforward conjecture regarding the distribution in steady-state. Proving convergence to the steady-state distribution is an additional step which is not often undertaken. (Some exceptions are the works by Gamarnik and Zeevi \cite{gamzee06} and the recent work of Stolyar and Yudovina \cite{stoyud11}.) For example, to establish the convergence of steady-state distributions in \cite{sto04}, one has to show that the limits used for the diffusion scaling and the steady-state limit (i.e., time going to $\infty$) can be interchanged, which can be done using the results already established in \cite{sto04}. In parallel with the development of analyzing the heavy-traffic limits of queueing systems, Harrison \cite{har88} suggested directly approximating the stochastic arrival and service processes in queueing systems by Brownian motions. This theme was further developed by Laws \cite{law92}, and Kelly and Laws \cite{kellaw93}. The main idea in these papers is to study convex optimization problems suggested by flow conservation equations, along with the Brownian control problems. This theme has been influential in much of the work on heavy-traffic optimality of control policies for queueing systems.

Instead of establishing optimality of control policies, if one is simply interested in evaluating the performance of a particular policy, a common technique is to study the drift of an appropriate Lyapunov function in steady-state. Assuming that appropriate moments exist, which sometimes might be non-trivial to establish as in the work of Glynn and Zeevi \cite{glyzee08}, setting the drift equal to zero in steady-state immediately provides bounds on the steady-state moments of queue lengths. An early use of this technique was used by Kingman to derive his well-known bound on the expected waiting time in a $G/G/1$ queue \cite{kin62a}. This idea was pursued successfully by Kumar and Kumar \cite{kumkum94} and by Bertsimas, Paschalidis and Tsitsiklis \cite{berpastsi94} who present many extensions of the basic idea to different types of queueing systems. The method was extended to loss models by Kumar, Srikant and Kumar in \cite{kumsrikum98}. The study of Lyapunov drift to analyze performance has its roots in Markov chain stability theory using the Foster-Lyapunov theorem (see the books by Asmussen \cite{asm03}, Meyn \cite{mey07}, and Meyn and Tweedie \cite{meytwe93}). An explicit connection between moment bounds and Lyapunov drift-based stability was provided by Kumar and Meyn in \cite{kummey95}. While most of these papers obtain bounds on polynomial moments, Hajek \cite{haj82} and Bertsimas, Gamarnik and Tsitsiklis \cite{bergamtsi01} obtained exponential-type bounds on the queue lengths. The bounds obtained by Hajek will be very useful to us in this paper.

Now that we have described prior work on heavy-traffic analysis and performance evaluation of queueing systems, we present our motivation for this paper. The Lyapunov drift-based moment bounding techniques are simple to derive since they require elementary probabilistic tools. Although the arrival and service processes are assumed to be quite simple to apply these techniques, the queueing models and control techniques can be quite complicated, thus making the techniques fairly general in their applicability. However, in many cases, the bounds obtained from these techniques are extremely loose in complex queueing systems such as wireless networks operating under the MaxWeight policy, or sometimes even in simple systems such as parallel servers to which arrivals are routed according to the JSQ policy. Therefore, the bounds obtained from drift considerations are sometimes not very useful to evaluate control policies for queueing systems. The main reason is that, simple drift-based bounds do not exploit resource pooling effects observed in heavy-traffic. The key to introducing resource pooling effects into the drift-based arguments is to define an appropriate notion of state-space collapse that can be easily incorporated into the derivation of the drift-based moment bounds. In this paper, we present techniques for doing so, i.e., we present techniques for introducing the notion of state-space collapse into the drift-based moment bound derivations.  In particular, we obtain upper bounds on the expected value of weighted functions of queue lengths for queueing systems operating under certain control policies, which coincide with lower bounds in heavy-traffic. Additionally, our techniques provide explicit bounds on moments even when the system is not in heavy-traffic, and can thus serve as a performance evaluation tool for the pre-limit systems when combined with optimization techniques such as those suggested in \cite{shawis08}.

We consider two types of queueing systems to illustrate our methodology: parallel servers where arrivals are routed according to the JSQ policy and wireless networks operating under the MaxWeight policy. In this paper, we call these two problems, the routing problem and the scheduling problem, respectively. The scheduling problem is directly motivated by scheduling in wireless networks and high-speed networks, while the routing problem is an abstraction of multi-path routing in communication networks, wireless or wireline. The proofs of results are simpler for the JSQ problem, and so for ease of exposition, we illustrate all the steps in our derivation using JSQ first. The MaxWeight policy is harder to analyze, but once the basic idea behind the proof is presented for JSQ, the extension to MaxWeight follows when the geometrical insight in Stolyar's work \cite{sto04} is translated into the Lyapunov drift framework. We note that the work of Gans and van Ryzin \cite{ganvan97,ganvan98} is similar in spirit to our work. They study heavy-traffic optimality in terms of steady-state moments. They first obtain a lower bound on the workload in a $G/G/1$ queue, which is similar to the Kingman-type bound \cite{kin62a} that we also use. However, beyond these similarities, their work is quite different in the following aspects: their upper bounds are derived for policies which make explicit use of the knowledge of the arrival rates and the topology of the capacity region (which is unknown in wireless networks with time-varying channels as we will see later). In contrast, the policies that we study do not require any of this information, and further, our Lyapunov drift-based approach appears to be quite different from the techniques used in \cite{ganvan97,ganvan98}. Also related is the work of Gupta and Shroff \cite{gupshr10}, who use Lyapunov-based bounds to numerically study the performance of MaxWeight algorithms. However, their bounds are not provably tight although they perform well in numerical studies and simulations. Thus, they cannot be used to establish heavy-traffic optimality.

\subsection{Outline of the Methodology}
\label{sec:MethodologyOutline}

The main contribution of the paper is a Lyapunov-drift based approach to obtaining bounds on steady-state queue lengths that are tight in heavy traffic. Our approach consists of three major steps:

\begin{enumerate}
\item Lower bound: First, we present a lower bound for the expected queue lengths in a single-server queue, and use this to identify appropriate lower bounds for both the parallel server model and the wireless network model. The lower bound that we obtain for the single-server queue is perhaps well-known and it follows the basic idea behind the Kingman bound \cite{kin62a}. However, we have not seen it explicitly stated, so we provide it here along with a short derivation.

\item State-space collapse: The next step is to show state-space collapse to a single dimension for both models. Unlike fluid limit proofs, in our model, the state does not actually collapse to a single dimension. What we show is that compared to the queue lengths, the differences between appropriately weighted queue lengths is small in an expected sense in steady-state.

\item Upper bound: The final step is to derive an upper bound on the expected steady-state queue lengths. For this step, as we will see, we use a natural Lyapunov function suggested by the resource pooling to be expected in the two problems. However, setting the drift of the Lyapunov function equal to zero directly does not yield a good upper bound. In addition, one has to use the state-space collapse result from the previous step to get an asymptotically tight upper bound. Once the upper bound is derived, checking that the upper and lower bounds coincide in heavy-traffic is straightforward. Even when the system is not in heavy-traffic, the upper bounds hold and can be of independent interest for performance analysis.
\end{enumerate}

\section{Notation, System Models and Other Preliminaries}
\label{sec:model}

In this section, we introduce some notation that will be used throughout the paper.
We consider the control of a network of $L$ queues that synchronously evolve in a time-slotted fashion. The evolution of the length of queue $l$ is governed by:
 \beqa
 Q_l[t+1] &=& \left(Q_l[t]+A_l[t]-S_l[t]\right)^+ \nonumber \\
 &=& Q_l[t] + A_l[t] - S_l[t] + U_l[t], \quad \textrm{for each } l=1,\cdots,L, \label{eqn:Qevolve}
 \eeqa
where $A_l[t]$ and $S_l[t]$ respectively denote the amount of
arrivals and offered services to queue $l$ in slot $t$, and $U_l[t]
\triangleq \max(0,S_l[t]-A_l[t]-Q_l[t])$ denotes the unused service
by queue $l$ in slot $t.$ We assume that $A_l[t]$ and $S_l[t]$ are
non-negative integer valued so that $Q_l[t],$ for each $l,$ evolves
over the space of non-negative integers. For convenience, we will
sometimes use boldface letters $\vQ, \vA, \vS,$ and $\vU$ to
represent the $L$-dimensional vectors of queue-lengths, arrivals,
offered services, and unused services, respectively. In our models
in this paper, the queue length process $\{\vQ[t]\}_{t\geq 0}$ will
form a Markov chain. We will say that the queueing system is stable
if this Markov chain is positive recurrent, and use $\bvQ$ to denote
the random vector whose probability distribution is the same as the
steady-state distribution of $\{\vQ[t]\}_{t\geq 0}$.

Since the queue-length vector $\vQ[t]$ evolves over integer values in the nonnegative quadrant of the $L$-dimensional real vector space $\bR_+^L,$ we will have occasion to use the following notation in the paper and so we state them once for convenience: for two vectors $\vx=(x_l)$ and $\vy=(y_l)$ in $\bR^L,$ their inner product, Euclidean norm, and the angle between them are respectively given by
 \beqa
 \La \vx,\vy \Ra \triangleq \sum_{l=1}^L x_l y_l,
    \quad \|\vx\|\triangleq \sqrt{\La \vx,\vx \Ra} = \sqrt{\sum_{l=1}^L x_l^2},
    \quad \theta_{\vx,\vy} \triangleq \arccos\left(\frac{\La \vx,\vy\Ra}{\|\vx\|\|\vy\|}\right). \label{eqn:RL_defs}
 \eeqa
We use $\preceq, \prec, \succeq$ to denote component-wise comparison of two vectors, and $\vZero$ and $\vOne$ to denote all-zero and all-one vectors, respectively. We call two vectors \emph{orthogonal,} denoted $\vx\perp \vy,$ if their inner product is zero, in which case the Pythagorean Theorem applies:
 \beqa
 \|\vx+\vy\|^2 = \|\vx\|^2 + \| \vy \|^2, \quad \textrm{for } \vx\perp\vy. \label{eqn:Pythagorean}
 \eeqa

Additional characteristics and constraints on the arrival and service processes depend on the particular problem we will consider, which will be discussed in Sections~\ref{sec:model_Routing} and \ref{sec:model_Scheduling} in the context of routing and scheduling, respectively. Next, we present a result that provides a bound on the steady-state moment-generating function of a random process using drift conditions.

\subsection{A Useful Result}

Our heavy-traffic analysis of the above models uses a result developed by Hajek \cite{haj82} in a more general context. In particular, this result will be useful in proving state collapse in the sense mentioned in Section~\ref{sec:MethodologyOutline}. We present it below for easy reference.

\begin{lemma} \label{lem:Hajek}
For an irreducible and aperiodic Markov Chain $\{X[t]\}_{t\geq 0}$ over a countable state space $\cX,$ suppose $Z:\cX\ra \bR_+$ is a nonnegative-valued Lyapunov function. We define the drift of $Z$ at $X$ as $$\Delta Z(X) \triangleq [Z(X[t+1])-Z(X[t])]\>\one(X[t]=X),$$ where $\one(.)$ is the indicator function. Thus, $\Delta Z(X)$ is a random variable that measures the amount of change in the value of $Z$ in one step, starting from state $X.$ This drift is assumed to satisfy the following conditions:
\begin{enumerate}
\item[(C1)] There exists an $\eta>0,$ and a $\kappa<\infty$ such that
 \beqano
    \bE[\Delta Z(X) | X[t]=X]
    \leq - \eta, \textrm{ for all } X\in \cX \textrm{ with } Z(X)\geq \kappa.
 \eeqano
\item[(C2)] There exists a $D < \infty$ such that
  \beqano
  \bP\left(|\Delta Z(X)| \leq D\right) = 1, \quad \textrm{ for all } X \in \cX.
  \eeqano
\end{enumerate}
Then, there exists a $\theta^\star > 0$ and a $C^\star < \infty$ such that
$$\limsup_{t\ra \infty} \bE\left[e^{\theta^\star Z(X[t])}\right] \leq C^\star.$$
If we further assume that the Markov Chain $\{X[t]\}_t$ is positive recurrent, then $Z(X[t])$ converges in distribution to a random variable $\bZ$ for which
$$ \bE\left[e^{\theta^\star \bZ}\right] \leq C^\star,$$
which directly implies that all moments of $\bZ$ exist and are finite.
\end{lemma}

We now introduce the two queueing control problems that will be considered in this paper.

\subsection{Routing Problem and the JSQ Routing Policy}
\label{sec:model_Routing}

In the routing problem, we consider a system of $L$ parallel servers and their associated queues. Packets are routed to one of the queues upon arrival (see Figure~\ref{fig:JSQ_System}). The goal is to find a routing policy which minimizes the total workload in the system.
\begin{figure}[!htbp]
\begin{center}
\includegraphics[width=3.1in]{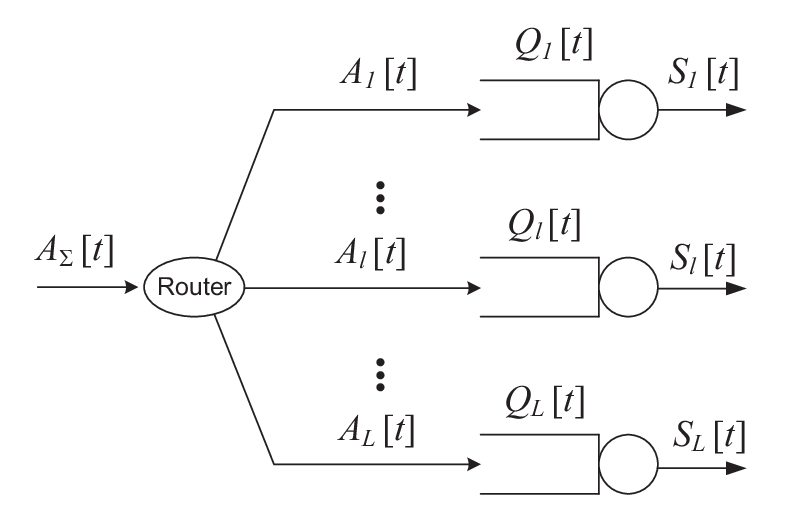}
\caption{System model for the routing problem.}
 \label{fig:JSQ_System}
 \end{center}
\end{figure}

We let $A_\Sigma[t]$ denote the random number of exogenous arrivals to the router at the beginning of slot $t.$ The router is then required to distribute the incoming packets to $L$ queues for service.
In other words, in each time slot $t$, the router is expected to select a nonnegative-valued vector $\vA[t] = (A_l[t])_l$ such that $\dsp \sum_{l=1}^L A_l[t] = A_\Sigma[t].$ Then, the queues will evolve as in (\ref{eqn:Qevolve}) based on this routing decision. We make the following assumptions on the arrival and service processes.

\begin{assumption}[Assumptions for the Routing Problem]\label{as:routing:AandS}
We assume that the exogenous arrival process and the service processes for different queues are independent (not necessarily identically distributed). Also, we assume that the exogenous arrival process $\{A_\Sigma[t]\}_{t}$ and each service process $\{S_l[t]\}_t$ is composed of a sequence of independent and identically distributed (i.i.d.) nonnegative-integer-valued and bounded random variables with $A_\Sigma[t]\leq A_{max}<\infty$ and $S_l[t] \leq S_{max}<\infty,$ for all $l$ and all $t.$
\end{assumption}

We additionally use the following notation for the mean and variance of different random variables in the routing problem: for the arrival process, $\lambda_\Sigma\triangleq\bE[A_\Sigma[1]]$,
$\sigma_\Sigma^2\triangleq var(A_\Sigma[1])$; for the service process of each queue $l$, $\mu_l \triangleq \bE[S_l[1]]$, $\nu^2_l\triangleq var(S_l[1])$; and for the hypothetical total service process, defined as $S_\Sigma[t]\triangleq \sum_{l=1}^L S_l[t],$ we let $\mu_\Sigma \triangleq \sum_{l=1}^L \mu_l,$ $\nu^2_\Sigma \triangleq \sum_{l=1}^L \nu^2_l.$ Finally, we also introduce the boldface notation for the $L$-dimensional vector of mean and variances for the service processes: $\vmu \triangleq (\mu_l)_l,$ and $\vnu^2 \triangleq (\nu_l^2)_l.$ Without loss of generality\footnote{We can eliminate any server with $\mu_l=0$ from the system since $S_l[t]=0$ for all $t$ with probability one for any such server.}, we assume that $\mu_{min}\triangleq\min_l \mu_l>0.$

We are interested in the performance of a well-known and very natural routing policy, called \emph{Join the Shortest Queue} (JSQ) Routing Policy, defined next.

\begin{definition}[JSQ Routing Policy]\label{def:JSQ}
For the routing problem introduced above (see Figure~\ref{fig:JSQ_System}), in each time slot $t,$ the \emph{Join-Shortest-Queue (JSQ) Routing Policy} forwards all incoming packets to one of the the queues with the shortest queue-length in that time-slot, breaking ties uniformly at random, i.e., given $\vQ[t],$ and $A_\Sigma[t],$ the arrival vector $\vA[t]$ is selected as
 \beqano
 \vA[t] &=& RAND\left\{\argmin_{\{\vA\geq \vZero: \sum_{l} A_l = A_\Sigma[t]\}} \La \vA, \vQ[t] \Ra \right\},
 \eeqano
where RAND denotes that ties are broken uniformly at random.
\end{definition}

The principle behind the JSQ routing policy is very intuitive: it constantly tries to equalize the queue-lengths at all servers by routing all the incoming packets in each time slot to one of the smallest queues. Thus, it hopes to make sure that no server is idle when there is work to be done. If this can be achieved, then the system will behave as though all the servers pool their resources together and act like a single server queue. In this paper, we make this claim precise by using a Lyapunov-based analysis.

It is clear that the maximum achievable service rate for the parallel queue system is $\mu_\Sigma = \sum_{l=1}^L \mu_l.$
A control policy is said to be throughput optimal if it can stabilize any set of arrival rates which can be stabilized by another policy.
The JSQ routing policy is well known to be throughput-optimal; additionally, the following lemma states that all moments are finite in steady-state. The proof is provided in Appendix~\ref{sec:JSQ_BddMoments} for completeness.

\begin{lemma}\label{lem:JSQ_BddMoments}
If the mean exogenous arrival rate $\lambda_\Sigma$ lies outside the capacity region $\cR,$ i.e., $\lambda_\Sigma > \mu_\Sigma,$ then the queueing network cannot be stabilized by any feasible routing policy.

For any $\lambda_\Sigma$ in the interior of $\cR,$ i.e., $\lambda_\Sigma < \mu_\Sigma,$ the JSQ Routing Policy stabilizes the network in the following strong sense: $\{\vQ[t]\}_t$ converges in distribution to a random vector $\bvQ$ whose all moments are bounded, i.e., there exist constants $\{M_r\}_{r=1,2,\cdots}$ such that $\bE[\|\bvQ\>\|^r] = M_r.$
\end{lemma}

\subsection{Scheduling Problem and the MW Scheduling Policy}
\label{sec:model_Scheduling}

In the scheduling problem, the goal is to select an instantaneous service rate vector for $L$ queues with independent packet arrival processes, subject to certain feasibility constraints on the rates at which the queues can be simultaneously served (see Figure~\ref{fig:MWS_System}).
\begin{figure}[!htbp]
\begin{center}
\includegraphics[width=3.5in]{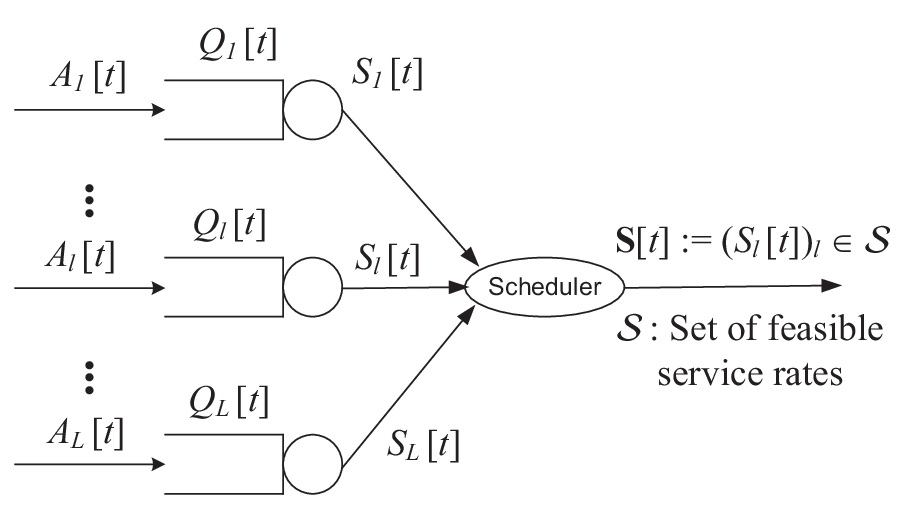}
\caption{System model for the scheduling problem.}
 \label{fig:MWS_System}
 \end{center}
\end{figure}

Here, the feasibility constraints can be used to model various types of coupling between service availability at different queues. For example, in the context of wireless networks with collision-based interference constraints, the feasibility constraints can capture interference among simultaneous wireless transmissions. Generically, we let $\cS$ denote the \emph{set of feasible service rate vectors} that the controller is allowed to select from. Hence, each element $\vS \triangleq (S_l)_l$ of $\cS$ is a vector of service rates that can be offered to the queues in one slot, and hence it contains the all zero vector $\vZero$, corresponding to no service.

In each time slot $t$, the scheduler is required to select a feasible service vector $\vS[t] \in \cS.$ Then, the queues will evolve according to (\ref{eqn:Qevolve}) based on this scheduling decision. We make the following assumptions on the arrival processes and the set of feasible service rates.

\begin{assumption}[Assumptions for the Scheduling Problem]\label{as:scheduling:AandS}
We assume that the arrival processes to different queues are independent (not necessarily identically distributed). Also, we assume that each exogenous arrival process $\{A_l[t]\}_{t}$ is composed of a sequence of i.i.d. nonnegative-integer-valued and bounded random variables with $A_l[t]\leq A_{max} <\infty$ for all $l$ and all $t.$
We assume that each feasible service rate vector $\vS \in \cS$ is nonnegative-integer-valued and bounded with $S_l \leq S_{max}<\infty,$ for all $l.$
\end{assumption}

We additionally use the following notation for the mean and variance of different processes in the scheduling problem: for the exogenous arrival process for each queue $l$, $\lambda_l\triangleq\bE[A_l[1]]$,
$\sigma_l^2\triangleq var(A_l[1])$; for the service process of each queue $l$, $\mu_l \triangleq \bE[S_l[1]]$, $\nu^2_l\triangleq var(S_l[1])$; for the hypothetical total arrival process, defined as $A_\Sigma[t]\triangleq \sum_{l=1}^L A_l[t],$ we let $\lambda_\Sigma \triangleq \sum_{l=1}^L \lambda_l,$ $\sigma^2_\Sigma \triangleq \sum_{l=1}^L \sigma^2_l.$ Finally, we also introduce the boldface notation for the $L$-dimensional vector of mean and variances for the arrival and service processes: $\vlam \triangleq (\lambda_l)_l,$ and $\vsig^2 \triangleq (\sigma_l^2)_l;$ and $\vmu \triangleq (\mu_l)_l,$ and $\vnu^2 \triangleq (\nu_l^2)_l.$
Without
loss of generality\footnote{We can eliminate any arrival with
$\lambda_l=0$ from the system since $A_l[t]=0$ for all $t$ with
probability one for any such arrival.}, we can assume that $\lambda_l>0,$ $\forall l.$

In the rest of paper, we are interested in the performance of a well-known queue-length-based scheduling policy, called \emph{Maximum Weight} (MW) Scheduling Policy, defined next.

\begin{definition}[MW Scheduling Policy]\label{def:MWS}
For the scheduling problem introduced above (see Figure~\ref{fig:MWS_System}), in each time slot $t,$ the \emph{Maximum Weight (MW) Scheduling Policy} selects the service rate vector $\vS[t]$ from $\cS$ to optimize the following objective, breaking ties uniformly at random,
 \beqano
 \vS[t] &=& RAND\left\{\argmax_{\vS \in \cS} \La \vQ[t], \vS \Ra\right\}
 \eeqano
where $\La \cdot,\cdot \Ra$ is the vector inner product in $\bR^L$ defined in (\ref{eqn:RL_defs}).
\end{definition}

\subsection{Capacity Region for the Scheduling Problem}

\begin{definition}[Maximum Achievable Rate (Capacity) Region for the Scheduling Problem, $\cR$] \label{def:SchedulingRateRegionR}
For the scheduling problem described in Section~\ref{sec:model_Scheduling} operating under Assumption~\ref{as:scheduling:AandS} and the given set of feasible service rate vectors $\cS,$ the maximum achievable rate region $\cR$ is the convex hull\footnote{The convex hull of a set $\cS \subset \bR^L$ is the smallest convex set that contains $\cS.$} of $\cS,$ i.e.,
 $$ \cR \triangleq \textrm{Convex Hull}(\cS).$$
Under the assumed finite size and nonnegative nature of the set $\cS,$ the region $\cR$ becomes a polyhedron in the nonnegative quadrant of $\bR^L$, and hence can be equivalently described by
 \beqa
 \cR = \{\vr \geq \vZero: \La \vc\uk, \vr \Ra \leq b\uk, \> k=1,\cdots,K\}, \label{eqn:SchedulingRateRegionR}
 \eeqa
where $K$ denotes the finite (and minimal) number of hyperplanes that fully describe the polyhedron, and the pair $(\vc\uk,b\uk)$ defines the $k^{th}$ hyperplane, $\cH\uk,$ through its normal vector $\vc\uk \in \bR^L$ and its inner product value $b\uk \in \bR$ as $\cH\uk \triangleq \{\vr: \La \vc\uk,\vr \Ra = b\uk\}.$ Further, since the allowed feasible service rates are nonnegative-valued, the pairs $(\vc\uk,b\uk)_k,$ are assumed to satisfy:
 $$ \|\vc\uk\| = 1, \qquad \vc\uk \succeq \vZero,
                \qquad b\uk > 0, \qquad \textrm{for } k=1,\cdots, K.$$
Furthermore, we call the intersection of the $k^{th}$ hyperplane with $\cR$ the $k^{th}$ \emph{face} $\cF\uk$ of the achievable rate region, i.e., $\cF\uk \triangleq \cH\uk \cap \cR.$
\hfill

\end{definition}

Figures~\ref{fig:MWSRateRegion2D} and \ref{fig:MWSRateRegion3D} illustrate the above concepts for a $2$-dimensional and a $3$-dimensional capacity region.
\begin{figure}[!htbp]
\begin{minipage}{3.1in}
\begin{center}
\includegraphics[width=2.3in]{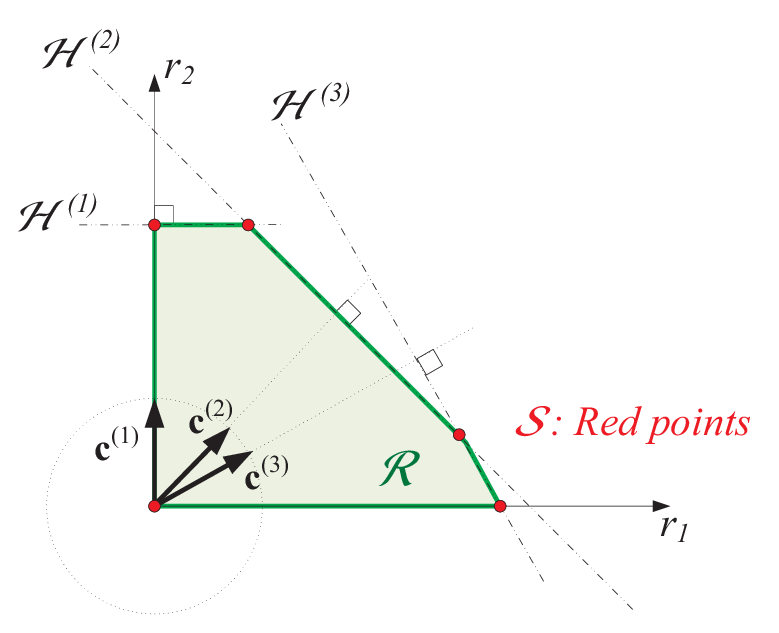}
\caption{The figure illustrates the capacity region for a $2$-dimensional example, indicating the hyperplanes $\cH\uk$ and the associated normal vectors $\vc\uk.$}
 \label{fig:MWSRateRegion2D}
 \end{center}
\end{minipage}
\hfill
\begin{minipage}{3.2in}
\begin{center}
\includegraphics[width=2.5in]{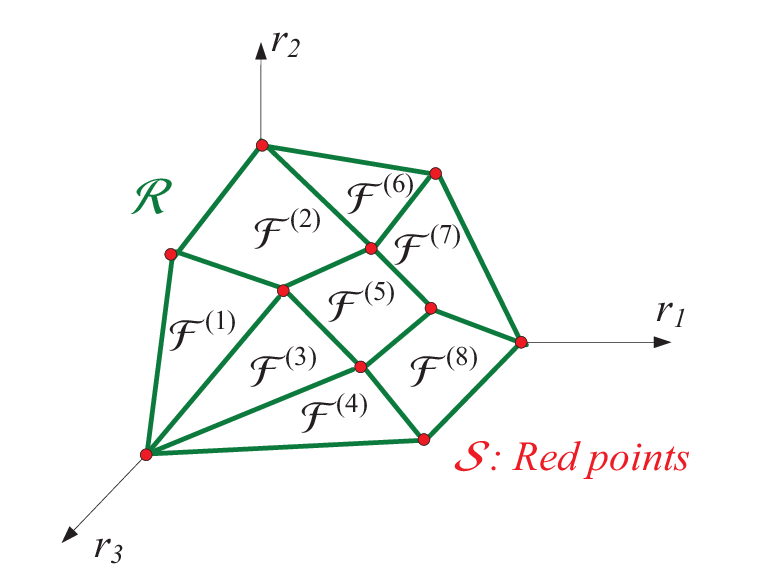}
\caption{The figure illustrates the capacity region for a $3$-dimensional example ($L=3$), indicating that each $\cF\uk$ is polyhedral in $\bR^{L-1}$ with a non-empty interior.}
 \label{fig:MWSRateRegion3D}
 \end{center}
\end{minipage}
\end{figure}
The next lemma states the well-known throughput-optimality of MW Scheduling, and shows that the steady-state moments of the queue length vector are bounded. As in the routing problem, the proof of the following lemma uses Lemma~\ref{lem:Hajek} and is omitted.

\begin{lemma} \label{lem:MWS_BddMoments}
If the arrival rate vector $\vlam \succeq \vZero$ lies outside the capacity region $\cR,$ i.e., $\vlam \notin \cR,$ then the queueing network cannot be stabilized by any feasible scheduling policy.

For any arrival rate vector $\vlam$ in the \emph{interior} of $\cR,$ i.e., $\vlam \in Int(\cR),$ the MW Scheduling Policy stabilizes the network in the following strong sense that the queue-length vector process $\{\vQ[t]\}_t$ converges in distribution to a random vector $\bvQ$ whose all moments are bounded, i.e., there exist finite numbers $\{M_r\}_{r=1,2,\cdots}$ such that $\bE[\|\bvQ\>\|^r] = M_r.$
\end{lemma}

\section{Lower Bounds}
\label{sec:LBs}

We first consider a simple single-server queueing system, depicted in Figure~\ref{fig:LBqueue} and obtain a lower bound on the steady-state expected queue length in this system. This lower bound will then be used to obtain lower bounds for both the routing and scheduling problems.

\begin{figure}[!htbp]
\begin{center}
\includegraphics[width=2.2in]{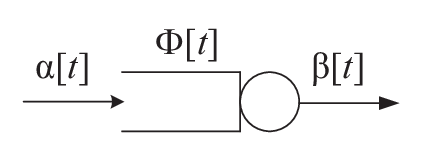}
\caption{Lower bounding system with i.i.d. arrival and service processes over time with bounded support.}
 \label{fig:LBqueue}
 \end{center}
\end{figure}
We assume that the arrival and service processes to the single-server queue are, respectively, described by two independent sequences of i.i.d. nonnegative-valued random variables $\{\alpha[t]\}_t$ and $\{\beta[t]\}_t.$ We assume that both distributions have finite support, i.e., there exist $\alpha_{max}<\infty$ and $\beta_{max}<\infty$ such that $\alpha[t]\leq \alpha_{max}$ and $\beta[t]\leq\beta_{max}$ with probability $1$ for all $t.$ We also denote the means and variances of the arrival and service processes as $\alpha\triangleq\bE[\alpha[1]],$ $\sigma_\alpha^2 \triangleq var(\alpha[1])$, and $\beta\triangleq\bE[\beta[1]],$ $\nu_{\beta}^2\triangleq var(\beta[1]).$ Then, the queue-length of the server $\Phi[t]$ evolves as
  \beqa
  \Phi[t+1] = \left(\Phi[t] + \alpha[t] - \beta[t]\right)^+, \qquad t\geq 0. \label{eqn:LB_Qevolve}
  \eeqa
We are now ready to provide a useful lower bound on the steady-state performance of this system.

\begin{lemma} \label{lem:LBs}
For the system of Figure~\ref{fig:LBqueue} with a given service process $\{\beta[t]\}_t$, consider the arrival process $\{\alpha\e[t]\}_t,$ parameterized by $\epsilon>0,$ with mean $\alpha\e$ satisfying $\epsilon = \beta-\alpha\e$, and with variance denoted as $\sigma_{\alpha\e}^2$. Let the queue-length process, denoted by $\{\Phi\e[t]\}_t$, evolve as in (\ref{eqn:LB_Qevolve}) with $\alpha[t]:=\alpha\e[t]$.

Then, $\{\Phi\e[t]\}_t$ is a positive Harris recurrent Markov Chain (\cite{meytwe93}) for any $\epsilon>0,$ and converges in distribution to a random variable $\bPhi\e$ with all bounded moments.

Moreover, the mean of $\bPhi\e$ can be lower-bounded  as
 \beqa
 \bE\left[\bPhi\e\right] &\geq& \frac{\zeta\e}{2\epsilon}- B_1, \label{eqn:LB_firstmoment}
 \eeqa
where $\zeta\e\triangleq \sigma_{\alpha}^2 + \nu_\beta^2 + \epsilon^2,$ and $B_1 \triangleq \frac{\beta_{max}}{2}.$

Therefore, in the \textbf{heavy-traffic limit} as the mean arrival rate approaches the mean service rate from below, i.e., as $\epsilon \downarrow 0,$ and assuming the variance $\sigma_{\alpha\e}^2$ converges to a constant $\sigma_\alpha^2$, the lower bounds become
 \beqa
 \liminf_{\epsilon \downarrow 0} \epsilon \bE\left[\bPhi\e\right] &\geq& \frac{\zeta}{2}, \label{eqn:LB_HT_firstmoment}
 \eeqa
 where $\zeta \triangleq \sigma_{\alpha}^2 + \nu_\beta^2.$
\end{lemma}

\begin{proof}
In the following argument, we will omit the $\e$ superscript for ease of exposition, and revive it when necessary.
The claim that $\{\Phi[t]\}_t$ is positive Harris recurrent follows from standard negative drift conditions (\cite{meytwe93}). Then, Lemma~\ref{lem:Hajek} applies to the Lyapunov function $V(\Phi)\triangleq\|\Phi\| $ as in the proofs of Lemmas~\ref{lem:JSQ_BddMoments} and \ref{lem:MWS_BddMoments}, and is omitted here. These establish that $\Phi[t]$ converges in distribution to $\bPhi$ with all bounded moments, i.e., $\bE[\|\bPhi\|^r] < \infty$ for each $r=1,2,\cdots.$

To prove the lower bound (\ref{eqn:LB_firstmoment}), we first expand (\ref{eqn:LB_Qevolve}) as
\beqa
 \Phi[t+1] = \Phi[t] + \alpha[t] - \beta[t] + \chi[t], \qquad t\geq 0. \label{eqn:LB_Qevolve2}
\eeqa
where $\chi[t]$ denotes the unused service in slot $t.$
For the quadratic Lyapunov function $W(\Phi)\triangleq \|\Phi\|^2,$ the mean drift  $\Delta W(\Phi) \triangleq [W(\Phi[t+1])-W(\Phi[t])]\>\one(\Phi[t]=\Phi)$ in steady-state must be zero, i.e., $\bE[\Delta W(\bPhi)]=0.$ Next, we expand the conditional mean drift of $W$, omitting the time reference $[t]$ for brevity:
\beqano
\bE[\Delta W(\Phi) \> | \> \Phi[t]=\Phi]
    &=& \bE[ (\Phi+\alpha-\beta+\chi)^2-\Phi^2  \> | \> \Phi] \\
    &=& \bE[ (\Phi+\alpha-\beta)^2+2 (\Phi+\alpha-\beta) \chi
                + \chi^2 - \Phi^2  \> | \> \Phi] \\
    &=& \bE[(\alpha-\beta)^2\>|\>\Phi]
            + 2 \bE[(\alpha-\beta)\>|\>\Phi] \Phi
                - \bE[\chi^2\>|\>\Phi]
\eeqano
where the last step uses the fact that $\chi (\Phi+\alpha-\beta) = -\chi^2$ by definition, and the independence of the arrival and service processes from each other, and from $\Phi.$ Taking expectations of both sides with respect to the steady-state distribution of $\{\Phi[t]\}$ and using the fact that $\bE[\Delta W(\bPhi)]=0$ yields, after re-organizing and reviving the $\e$ notation,
\beqano
\epsilon \bE[\bPhi\e] &=& \frac{\bE[(\alpha\e-\beta)^2]}{2} - \frac{\bE[\chi(\bPhi\e)^2]}{2}\\
    &\geq& \frac{\zeta\e}{2\epsilon}- \frac{\epsilon  \>\beta_{max}}{2},
\eeqano
where the last step follows from simple manipulations of the first term, and from the facts that $\bE[\chi(\bPhi\e)]=\epsilon,$ and $\chi[t]\leq \beta_{max}$ for all $t.$
\hfill
\end{proof}

Next, we will discuss the implications of this lower bound on the routing and scheduling problems, respectively.

\subsection{Lower Bounds for the Routing Problem}
\label{sec:LB_Routing}

To lower bound the performance of any feasible routing policy, we assume \emph{resource pooling} whereby we consider a hypothetical single server that serves all exogenous arrivals with a service rate of $\{S_\Sigma[t]\}$ defined as the sum of the service processes of all servers in the actual system, i.e., $S_\Sigma \triangleq \sum_{l=1}^L S_l[t].$
This results in a single-server queueing system of Figure~\ref{fig:LBqueue} with the arrival process $\alpha[t] = A_\Sigma[t]$, for all $t$ with $\alpha_{max} = A_{max},$ and the service process $\beta[t] = S_\Sigma[t]$, for all $t$ with $\beta_{max} = L S_{max}.$ It is then easy to see that the corresponding queue-length process $\{\Phi[t]\}_t$ is stochastically smaller than the total queue-length process $\{\sum_{l=1}^L Q_l[t]\}_t$ of the original multi-server system in Figure~\ref{fig:JSQ_System} under any feasible routing policy. Thus, utilizing the lower bounds from Lemma~\ref{lem:LBs}, we next establish lower bounds on the performance of any routing policy (and hence JSQ) for all arrival rates.

\begin{lemma}\label{lem:Routing_LBs}
For the routing problem of Section~\ref{sec:model_Routing} with a given service vector process $\{\vS[t]\}_t,$ consider the exogenous arrival process $\{A_\Sigma\e[t]\}_t$, parameterized by $\epsilon>0,$ with mean $\lambda_\Sigma\e$ satisfying $\epsilon = \mu_\Sigma-\lambda_\Sigma\e$, and with variance denoted as $(\sigma_\Sigma\e)^2$. Accordingly, let the queue-length process under JSQ Routing with this arrival process be denoted as $\{\vQ\e[t]\}_t$, evolving as in (\ref{eqn:Qevolve}). Moreover, let $\bvQ\e$ denote the limiting random vector with all bounded moments that the process $\{\vQ\e[t]\}_t$ is known (by Lemma~\ref{lem:JSQ_BddMoments}) to converge to, for any $\epsilon>0.$

Then, the first moment of the sum length of $\bvQ\e$ can be lower-bounded, with notation $\zeta\e \triangleq (\sigma_\Sigma\e)^2 + \nu_\Sigma^2 + \epsilon^2,$ as:
 \beqa
 \bE\left[\sum_{l=1}^L \bQ_l\e\right] &\geq& \frac{\zeta\e}{2\epsilon}- B_1 , \label{eqn:LB_firstmoment2}
 \eeqa
where $B_1 \triangleq \frac{L S_{max}}{2}.$

Therefore, in the \textbf{heavy-traffic limit} as the mean arrival rate approaches the mean service rate from below, i.e., as $\epsilon \downarrow 0,$ and as the variance $(\sigma_\Sigma\e)^2$ converges to a constant $\sigma_\Sigma^2$, the lower bound becomes, with notation $\zeta \triangleq \sigma_{\Sigma}^2 + \nu_\Sigma^2,$
 \beqa
 \liminf_{\epsilon \downarrow 0} \epsilon \bE\left[\sum_{l=1}^L \bQ_l\e\right] &\geq& \frac{\zeta}{2}. \label{eqn:LB_HT_firstmoment1}
 \eeqa
\end{lemma}

\subsection{Lower Bounds for the Scheduling Problem}
\label{sec:LB_Scheduling}

Unlike the routing problem, where there is a single natural lower bound on the total queue-length in the system, in the case of scheduling, the polyhedral nature of the capacity region $\cR$ (see Definition~\ref{def:SchedulingRateRegionR}), implies that there are many possible lower bounds.

From Definition~\ref{def:SchedulingRateRegionR}, recall that the capacity region is bounded by $K$ hyperplanes, each hyperplane $\cH\uk$ described by its normal vector $\vc\uk$ and the value $b\uk.$ Throughout the rest of the paper, we will consider a particular face $\cF\uk$ and a point $\vlam\uk \in Relint(\cF\uk),$ where $Relint(\cF\uk)$ is the \emph{relative interior} of the polyhedral set $\cF\uk$.  Consider an arrival rate vector $\vlam\e$ parameterized by $\epsilon>0$ such that
 \beqa
  \vlam\e &\triangleq& \vlam\uk - \epsilon \vc\uk, \label{eqn:projection_k}
  \eeqa
so that $\vlam\e \in Int(\cR).$ In other words, $\vlam\e$ is a stabilizable rate in the capacity region that is at a distance $\epsilon$ away from the $k^{th}$ face $\cF\uk.$ In the remainder of this section, we will derive a lower bound on the steady-state weighted queue-length $\La \vc\uk, \vQ[t]\Ra$ as a function of $\epsilon.$ To that end, we construct the following hypothetical single-server queue associated with the $k^{th}$ face: following the notation in Figure~\ref{fig:LBqueue} with the superscript $\uk$ to indicate the $k^{th}$ face, we set $\alpha\uk[t] = \La \vc\uk, \vA[t] \Ra,$ where $\vA[t]$ is the vector of arrivals to the $L$ queues in the scheduling problem, and $\beta\uk[t] = b\uk,$ where $b\uk$ is the positive constant defining the value of the $k^{th}$ hyperplane.
Accordingly, the single server queue-length $\Phi\uk[t]$ evolves as in (\ref{eqn:LB_Qevolve}) with these arrival and service process mappings.
Then, $\{\Phi\uk[t]\}_t$ is stochastically smaller than the queue-length process $\{\La \vc\uk, \vQ[t] \Ra\}_t$ under any feasible (and hence MW) scheduling policy. Hence, utilizing Lemma~\ref{lem:LBs}, we can find the following lower bounds on the moments of the limiting queue-length vector under MW Scheduling.

\begin{lemma}\label{lem:Scheduling_LBs}
For the scheduling problem of Section~\ref{sec:model_Scheduling} with a given set of feasible schedules $\cS,$ consider the exogenous arrival vector process $\{\vA\uve[t]\}_t$ with mean vector $\vlam\uve \in Int(\cR)$ defined in (\ref{eqn:projection_k}) and with variance vector denoted as $(\vsig\uve)^2\triangleq \left((\sigma_l\uve)^2\right)_{l=1}^L$. Let the queue-length process under MW Scheduling with this arrival process be denoted as $\{\vQ\uve[t]\}_t$, evolving as in (\ref{eqn:Qevolve}). Moreover, let $\bvQ\uve$ denote a random vector with the same distribution as the steady-state distribution of $\{\vQ\uve[t]\}_t$ (by Lemma~\ref{lem:MWS_BddMoments}).

Then, for the given $k,$ and with $\zeta\uvek \triangleq \La (\vc\uk)^2,(\vsig\uve)^2\Ra + (\euk)^2 = \dsp \sum_{l=1}^L (c_l\uk)^2 (\sigma_l\uve)^2+ (\euk)^2 ,$
 \beqa
 \bE\left[\La \vc\uk, \bvQ\uve \Ra\right] &\geq& \frac{\zeta\uvek}{2\euk}- B_1\uk \label{eqn:Scheduling_LB_firstmoment}
 \eeqa
where $B_1\uk \triangleq \frac{b\uk}{2}.$

Further, consider the \textbf{heavy-traffic limit} $\euk\downarrow 0;$ and suppose that the variance vector $(\vsig\uve)^2$ approaches a constant vector $\vsig^2.$ Then, defining $\zeta\uk \triangleq \La (\vc\uk)^2,\vsig^2\Ra,$ we have
 \beqa
 \liminf_{\euk\downarrow 0}
    \euk\bE\left[\La \vc\uk, \bvQ\uve \Ra\right]
        &\geq& \frac{\zeta\uk}{2}
            \label{eqn:Scheduling_LB_HT_firstmoment}
 \eeqa

\end{lemma}

\section{State-Space Collapse}
\label{sec:Collapse}

In this section, we will show a moment condition which is the equivalent of state-space collapse in prior literature of fluid and diffusion limit analysis of the routing and scheduling problems. Specifically, we will show that the steady-state queue-length vector concentrates around a line within the $L$-dimensional state space in the following sense: the deviations from the line are bounded, independent of heavy-traffic parameter $\epsilon.$ Since the lower bounds on the queue lengths are of the order of $1/\epsilon,$ this would then show that the deviations from the line are small compared to the queue lengths themselves.
To establish state-space collapse, in both problems, we will first identify a vector and show that the queue lengths concentrate along this vector.
For this purpose, we first introduce some notation.

Let $\vc \succeq \vZero$ be a vector with unit norm. The components of the queue length vector $\vQ$ parallel and perpendicular to $\vc$ are, respectively, given by:
 \beqano
 \vQ_\parallel &\triangleq& \La \vc,\vQ \Ra \> \vc, \\
 \vQ_\perp &\triangleq& \vQ - \vQ_\parallel \> = \> \left[ Q_l - \La \vc,\vQ \Ra \> c_l\right]_{l=1}^L.
 \eeqano
If $\vc$ represents the direction along which state-space collapse occurs, then we study the statistics of $\vQ_\perp$ to show that all its moments are bounded by constants that do not depend on the proximity of the arrival rate vector to the boundary of the capacity region. To that end, we utilize Lemma~\ref{lem:Hajek} by studying the drift of the Lyapunov function $V_\perp(\vQ)=\|\vQ_\perp\|$ to show that: (i) the drift of this Lyapunov function is always bounded; (ii) and that, when $\|\vQ_\perp\|$ value is sufficiently large, it has a strictly negative drift \emph{whose value does not depend on the distance of the arrival rate vector from the boundary of the capacity region}, denoted as $\epsilon$ and $\veps$ in Sections~\ref{sec:LB_Routing} and \ref{sec:LB_Scheduling}. This latter property will be important in establishing the heavy-traffic optimality of the schemes. It turns out that studying of the drift of $\|\vQ_\perp\|$ is hard; instead, as we will see later, it is easier to study the drift of $\|\vQ\|^2$ and $\|\vQ_\parallel\|^2$. The following lemma serves as a useful preliminary step which relates the drift of these three quantities; the proof is provided in Appendix~\ref{sec:CollapseFirstSteps}. We will use this result in deriving our state-space collapse results later.

\begin{lemma}\label{lem:CollapseFirstSteps}
Consider the generic queueing system described in Section~\ref{sec:model} with queues evolving according to (\ref{eqn:Qevolve}), where the arrival and service processes are respectively bounded by $A_{max}$ and $S_{max}$ for each link and each time slot. Define the following Lyapunov functions:
 \beqa
 V_\perp(\vQ)\triangleq \|\vQ_\perp\|, \quad W(\vQ)\triangleq \|\vQ\|^2, \quad \textrm{and} \quad W_\parallel(\vQ)\triangleq \|\vQ_\parallel\|^2,
 \label{eqn:LyapDefs}
 \eeqa
with the corresponding single-step drifts denoted by:
 \beqa
 \Delta V_\perp(\vQ) &\triangleq& [V_\perp(\vQ[t+1])-V_\perp(\vQ[t])]\>\one(\vQ[t]=\vQ), \label{eqn:LyapDriftDefV_perp}\\
 \Delta W(\vQ) &\triangleq& [W(\vQ[t+1])-W(\vQ[t])]\>\one(\vQ[t]=\vQ),\label{eqn:LyapDriftDefW}\\
 \Delta W_\parallel(\vQ) &\triangleq& [W_\parallel(\vQ[t+1])-W_\parallel(\vQ[t])]\>\one(\vQ[t]=\vQ).\label{eqn:LyapDriftDefW_par}
 \eeqa
Then, the following properties hold for the random variable $\Delta V_\perp(\vQ)$:
\begin{enumerate}
\item We can bound the drift of $V_\perp$ with the drifts of $W$ and $W_\parallel$ as follows:
 \beqa
 \Delta V_\perp(\vQ)&\leq& \frac{1}{2 \|\vQ_\perp\|} \left(\Delta W(\vQ)-\Delta W_\parallel(\vQ)\right), \qquad \textrm{for all} \quad \vQ \in \bR_+^L. \label{eqn:VperpDriftBound}
 \eeqa
\item The drift of $V_\perp$ is absolutely bounded as:
 \beqa
 |\Delta V_\perp(\vQ) | &\leq & 2 \sqrt{L} \max(A_{max},S_{max}),  \qquad \textrm{for all} \quad \vQ \in \bR_+^L,
 \label{eqn:VperpMaxBound}
 \eeqa
where, recall from Section~\ref{sec:model} that, $A_{max}$ and $S_{max}$ respectively bounds the number of arrivals to and departures from a queue, and $L$ is the number of queues in the network.
\end{enumerate}
\end{lemma}

Notice that (\ref{eqn:VperpMaxBound}) verifies Condition $(C2)$ of Lemma~\ref{lem:Hajek} with
$$D=2 \sqrt{L} \max(A_{max},S_{max})$$
for the Lyapunov function $V_\perp(\vQ).$ To conclude that $V_\perp(\vQ[t])$ converges to $\|\bvQ_\perp\|$ with finite moments, we also need to verify condition $(C1)$ of Lemma~\ref{lem:Hajek}. In the following subsections, we do so for both the JSQ Routing and MW Scheduling policies. Further, we will show that the moments bounds on $\|\bvQ_\perp\|$ are independent of the proximity of the arrival rate vectors to the boundary of the capacity regions -- a key feature to be used for the proof of heavy-traffic optimality of these policies in Section~\ref{sec:UBs}.

\subsection{State-Space Collapse for the Routing Problem under the JSQ Policy}
\label{sec:Collapse_JSQ}

We recall the model introduced in Lemma~\ref{lem:Routing_LBs} for the routing problem, where the exogenous arrival process $\{A_\Sigma\e\}_t$ is parameterized by a positive valued $\epsilon \triangleq \mu_\Sigma - \lambda\e_\Sigma,$ and where the queue-length process and the limiting queue-length random vector under the JSQ Router are respectively denoted by $\{\vQ\e[t]\}_t$ and $\bvQ\e.$ Recall that the JSQ policy attempts to equalize the queue lengths, so one would expect the state-space to collapse along the direction of a unit vector, all of whose components are equal, i.e., the vector $\vc \triangleq \frac{\vOne}{\sqrt{L}} = (\frac{1}{\sqrt{L}})_{l=1}^L.$ Then, the projection and the perpendicular of any given $\vQ\e \in \bR_+^L$ with respect to this line become:
 \beqano
 \vQ_\parallel\e \> \triangleq \> \frac{Q_\Sigma\e}{L} \vOne, \qquad \textrm{and} \qquad
 \vQ_\perp\e \> \triangleq \> \left[Q_l-\frac{1}{L}Q_\Sigma\e\right]_{l=1}^L,
 \eeqano
where $Q_\Sigma\e \triangleq \sum_{l=1}^L Q_l\e$ denotes the total number of packets in the network. We already know, from Lemma~\ref{lem:JSQ_BddMoments}, that $\vQ\e[t]$ converges to $\bvQ\e$ in distribution for any $\epsilon>0$. Due to continuous mapping theorem, this implies that $\vQ_\parallel\e[t]$ also converges in distribution to $\bvQ_\parallel\e.$ The following result establishes the boundedness of all moments of the limiting $\bvQ_\perp\e$ uniformly for all $\epsilon>0.$

\begin{proposition}\label{prop:JSQ_Qperp_boundedness}
Consider the limiting queueing process $\bvQ\e$ under JSQ Routing, serving the exogenous arrival process $\{A_\Sigma\e\}_t$ parameterized by $\epsilon \triangleq \mu_\Sigma - \lambda\e_\Sigma.$
Then, for any choice of $\delta \in (0,\mu_{min})$, where
$\mu_{min}\triangleq \min_l \mu_l$, there exists a sequence of
finite numbers $\{N_r\}_{r=1,2,\cdots}$ such that
$\bE\left[\|\bvQ_\perp\e\|^r\right] \leq N_r,$ for all $\epsilon \in
(0,(\mu_{min}-\delta)L)$ and for each $r=1,2,\cdots.$
\end{proposition}

\begin{proof}
The proof follows from Lemma~\ref{lem:Hajek} applied to the Lyapunov function $V_\perp(\vQ\e)\triangleq \|\vQ_\perp\e\|,$ where the following argument omits the superscript $\e$ for ease of exposition. Note that (\ref{eqn:VperpMaxBound}) already establishes Condition (C2) of Lemma~\ref{lem:Hajek}. Thus, the proof is done once we verify Condition (C1) for the JSQ policy. To that end, we use the bound on $\Delta V_\perp(\vQ)$ in (\ref{eqn:VperpDriftBound}) and study the mean conditional drift of $\Delta W(\vQ)$ and $\Delta W_\parallel(\vQ)$ next.

We start by upper-bounding
$\bE[\Delta W(\vQ)\> | \> \vQ]\triangleq\bE[\Delta W(\vQ)\> | \> \vQ[t]=\vQ]$. In what follows, we will omit the time reference $[t]$ after the first step for brevity.
\beqa
 \bE[\Delta W(\vQ)\> | \> \vQ]
     &=& \bE\left[\|\vQ[t+1]\|^2 - \|\vQ\|^2 \> | \> \vQ\right]\nonumber\\
     &=& \bE\left[\|\vQ+\vA-\vS+\vU\|^2 - \|\vQ\|^2 \> | \> \vQ\right]\nonumber\\
     &=& \bE\left[\|\vQ+\vA-\vS\|^2 +2 \La \vQ+\vA-\vS, \vU \Ra +\|\vU\|^2 - \|\vQ\|^2 \> | \> \vQ\right]\nonumber\\
     &\stackrel{(a)}{\leq}& \bE\left[\|\vQ+\vA-\vS\|^2 - \|\vQ\|^2 \> | \> \vQ\right]\nonumber\\
     &=& \bE\left[2 \La \vQ, \vA-\vS \Ra + \|\vA-\vS\|^2 \> | \> \vQ\right]\nonumber\\
     &\leq& 2 \>\bE\left[ \La \vQ, \vA-\vS \Ra \> | \> \vQ\right] + K_1, \label{eqn:JSQcollapse:aux}
\eeqa
where the inequality (a) follows from the fact that $U_l(Q_l+A_l-S_l)=-U_l^2\leq 0,$ for each $l,$
and $K_1 \triangleq L \max(A_{max},S_{max})^2$ is bounded  since both the arrival and service processes are bounded.

Next, we bound the first term in (\ref{eqn:JSQcollapse:aux}) by first defining a hypothetical arrival rate vector $\vlam = (\lambda_l)_l$ with respect to the given service rate vector $\vmu$ and the $\epsilon \in (0,(\mu_{min}-\delta)L)$ as $\vlam \triangleq \vmu - \frac{\epsilon}{\sqrt{L}} \vc.$ Note that $\lambda_{min} \triangleq \min_l \lambda_l > \delta$, where $\delta$ is a fixed constant in $(0,\mu_{min}),$ and that $\sum_{l=1}^L \lambda_l = \mu_\Sigma - \epsilon,$ which is identical to the assumed rate $\lambda_\Sigma$ of the exogenous arrival process $\{A_\Sigma[t]\}_t$.
The following derivation starts by adding and subtracting $\vlam$ to the first term in (\ref{eqn:JSQcollapse:aux}):
\beqano
\bE\left[ \La \vQ, \vA-\vS \Ra \> | \> \vQ\right]
    &=& \La \vQ, \bE\left[\vA\> | \> \vQ\right]-\vlam\Ra-\La \vQ, \vmu-\vlam \Ra\\
    &\stackrel{(a)}{=}&
        \La \vQ, \bE\left[\vA\> | \> \vQ\right]-\vlam\Ra-\frac{\epsilon}{\sqrt{L}}\La \vQ, \vc \Ra\\
    &\stackrel{(b)}{=}& \bE[A_\Sigma \> | \> \vQ] \> Q_{min} - \La \vQ, \vlam \Ra -\frac{\epsilon}{\sqrt{L}} \|\vQ_\parallel\|\\
    &=& \lambda_\Sigma \> Q_{min} - \sum_{l=1}^L \lambda_l Q_l -\frac{\epsilon}{\sqrt{L}} \|\vQ_\parallel\|\\
    &=& -\sum_{l=1}^L \lambda_l (Q_l-\> Q_{min})
        -\frac{\epsilon}{\sqrt{L}} \|\vQ_\parallel\|\\
    &\leq& - \>\lambda_{min} \sum_{l=1}^L |Q_l-\> Q_{min}| -\frac{\epsilon}{\sqrt{L}} \|\vQ_\parallel\|\\
    &=& - \|\vQ-\> Q_{min} \> \vOne\|_1 \> \lambda_{min}
    -\frac{\epsilon}{\sqrt{L}} \|\vQ_\parallel\|\\
    &\stackrel{(c)}{\leq}& -\|\vQ-\> Q_{min} \> \vOne\| \> \lambda_{min}-\frac{\epsilon}{\sqrt{L}} \|\vQ_\parallel\|\\
    &\stackrel{(d)}{\leq}& -\|\vQ-\> \frac{1}{L} Q_\Sigma \> \vOne\| \> \lambda_{min}-\frac{\epsilon}{\sqrt{L}} \|\vQ_\parallel\|\\
    &\stackrel{(e)}{\leq}& - {\delta} \> \|\vQ_\perp\| -\frac{\epsilon}{\sqrt{L}} \|\vQ_\parallel\|
\eeqano
where
$$Q_{min} \triangleq \min_{1\leq m\leq L}  Q_m \geq 0;\qquad
\|\vx\|_1 \triangleq \sum_{l=1}^L |x_l| \textrm{ is the $l_1$-norm of } \vx,$$
and step (a) follows from the definition of $\vlam;$ (b) follows from the definitions of the JSQ policy (see Definition~\ref{def:JSQ}) and of $\vQ_\parallel;$ (c) is true since $l_1$ norm $\|\vx\|_1$ of any vector $\vx \in \bR^L$ is no smaller than its $l_2$ (Euclidean) norm $\|\vx\|;$ (d) is true since $\frac{1}{L} Q_\Sigma$ minimizes the convex function $\|\vQ-y \> \vOne\|$ over $y\in\bR;$ (e) is true since $\lambda_{min} > \delta$ by construction of $\vlam.$ This bound, when substituted in (\ref{eqn:JSQcollapse:aux}), yields:
\beqa
 \bE[\Delta W(\vQ)\> | \> \vQ] &\leq& - 2 \> \delta \> \|\vQ_\perp\|
 -\frac{2 \>\epsilon}{\sqrt{L}} \|\vQ_\parallel\| + K_1.
 \label{eqn:JSQ:WdriftBound}
\eeqa
Next, we lower bound $\bE[\Delta W_\parallel(\vQ)\> | \> \vQ]\triangleq\bE[\Delta W_\parallel(\vQ)\> | \> \vQ[t]=\vQ]$:
 \beqa
 \bE[\Delta W_\parallel(\vQ)\> | \> \vQ]
 &=& \bE\left[\La \vc, \vQ[t+1]\Ra^2 - \La \vc, \vQ[t]\Ra^2 \> | \> \vQ\right] \nonumber \\
 &=& \bE\left[\La \vc, \vQ+\vA-\vS+\vU\Ra^2 - \La \vc, \vQ\Ra^2 \> | \> \vQ\right] \nonumber \\
 &=& \bE\left[\La\vc,\vQ+\vA-\vS\Ra^2 + 2\La\vc,\vQ+\vA-\vS\Ra \La\vc,\vU\Ra\right.\nonumber\\
 &&\left.+ \La\vc,\vU\Ra^2 - \La\vc,\vQ\Ra^2\> | \> \vQ\right] \nonumber \\
 &=& \bE\left[2\La\vc,\vQ\Ra \La\vc,\vA-\vS\Ra + \La\vc,\vA-\vS\Ra^2\right.
   \nonumber\\&&\left.+ 2\La\vc,\vQ+\vA-\vS\Ra \La\vc,\vU\Ra + \La\vc,\vU\Ra^2\> | \> \vQ\right]\nonumber \\
 &=& 2\La\vc,\vQ\Ra \La\vc,\bE\left[\vA\> | \> \vQ\right]-\vmu\Ra
        - 2\bE\left[\La\vc,\vS\Ra \La\vc,\vU\Ra \> | \> \vQ\right] \nonumber\\
 & & + \bE\left[\La\vc,\vA-\vS\Ra^2 + 2 \La \vc, \vQ +\vA \Ra \La \vc, \vU \Ra
     + \La\vc,\vU\Ra^2\> | \> \vQ\right] \label{eqn:JSQcollapse:aux1} \\
 &\geq& 2\La\vc,\vQ\Ra \La\vc,\bE\left[\vA\> | \> \vQ\right]-\vmu\Ra
        - 2\bE\left[\La\vc,\vS \Ra \La\vc,\vU\Ra \> | \> \vQ\right] , \label{eqn:JSQcollapse:aux2}
 \eeqa
where the inequality follows from the observation that $(\ref{eqn:JSQcollapse:aux1}) \geq 0$ as $\vc \succ \vZero,$ $\vU \succeq \vZero,$ $\vQ \succeq \vZero.$ We further lower-bound the remaining terms in (\ref{eqn:JSQcollapse:aux2}) individually. For the first term, we first recall that $\La \vc, \vQ \Ra = \|\vQ_\parallel\|.$ Also noting that $\vc = \vOne/\sqrt{L}$ here and that $\lambda_\Sigma=\mu_\Sigma-\epsilon,$
\beqano
\La\vc,\bE\left[\vA\> | \> \vQ\right]-\vmu\Ra &=& \frac{1}{\sqrt{L}} \sum_{l=1}^L (\bE[A_l\> | \> \vQ] - \mu_l)\> = \> \frac{\lambda_\Sigma-\mu_\Sigma}{\sqrt{L}} \> =\> -\frac{\epsilon}{\sqrt{L}}.
\eeqano
Finally, we can lower-bound the second term in (\ref{eqn:JSQcollapse:aux2}) easily by $-K_2,$ where $K_2 \triangleq 2 L S_{max}^2,$ since $S_l \leq S_{max}$ and $U_l \leq S_{max}$ for all $l.$ Thus, using these lower bounds in (\ref{eqn:JSQcollapse:aux2}) yields:
\beqa
 \bE[\Delta W_\parallel(\vQ)\> | \> \vQ] &\geq& -\frac{2\> \epsilon}{\sqrt{L}} \|\vQ_\parallel\| - K_2. \label{eqn:JSQ:W_par_driftBound}
\eeqa
We now substitute the bounds (\ref{eqn:JSQ:WdriftBound}) and (\ref{eqn:JSQ:W_par_driftBound}) in (\ref{eqn:VperpDriftBound}), and cancel common terms, to bound the conditional mean drift of $V_\perp(\vQ)$ as
\beqano
\bE[\Delta V_\perp(\vQ)\>|\>\vQ] &\leq& -\delta + \frac{K_1+K_2}{2\|\vQ_\perp\|},
\eeqano
where $\delta,$ $K_1,$ and $K_2$ are positive constants defined above, all independent of $\epsilon.$ Note that this inequality verifies Condition (C1) of Lemma~\ref{lem:Hajek}, and hence establishes the claimed existence of finite constants $\{N_r\}_{r=1,2,\cdots}$ for which $\bE[\|\bvQ_\perp\e\|^r] \leq N_r,$ for all $\epsilon \in (0,(\mu_{min}-\delta)L),$ and for each $r=1,2,\cdots.$
\hfill
\end{proof}

This proposition establishes that even for a heavily loaded network, where $\epsilon$ approaches zero, the steady-state queue-length vector concentrates around the line along $\vc,$ and the difference between the queue lengths have bounded moments. In contrast, the lower bound on the total queue length goes to infinity as $\epsilon$ goes to zero. Thus, in comparison to the total queue length, the differences in the queue lengths are negligible. This observation will be critical in Section~\ref{sec:UBs} for arguing that the queueing network performs close to the lower-bounding resource-pooled system of Section~\ref{sec:LBs}. Since relative to the queue lengths, the differences in the queue lengths are negligible, the queue lengths can be thought as being attracted towards the line defined by the vector $\vc.$ We will call this line, the \emph{line of attraction}.

\subsection{State-Space Collapse for the Scheduling Problem under the MW Policy}
\label{sec:Collapse_MWS}

We recall the model introduced in Lemma~\ref{lem:Scheduling_LBs} for the scheduling problem, where the exogenous arrival vector process $\{\vA\uve[t]\}_t$ is parameterized by $\veps > 0$ that measures the distance of the arrival rate vector $\vlam\e\in Int(\cR)$ for the $k^{th}$ face of the achievable rate region. Also recall that we fix a face $\cF\uk$ and a $\lambda\uk\in Relint(\cF\uk)$ throughout. The corresponding queue-length vector process and the limiting queue-length random vector under the MW Scheduler are, respectively, denoted by $\{\vQ\uve[t]\}_t$ and $\bvQ\uve.$

Different from the routing case, in the scheduling problem, depending upon the arrival rate vector, the line of attraction can be different. In general, associated with each hyperplane enclosing the capacity region $\cR$ (cf. Definition~\ref{def:SchedulingRateRegionR}) is a possible line of attraction. Recall that the pair $(\vc\uk,b\uk)$ defines the $k^{th}$ hyperplane $\cH\uk$ through its unit norm normal vector $\vc\uk$ and its value $b\uk>0.$ Accordingly, we define the projection and the perpendicular vector of any given $\vQ\uve$ with respect to the $k^{th}$ normal as:
 \beqano
 \vQ_\parallel\uvek &\triangleq& \La \vc\uk,\vQ\uve \Ra \> \vc\uk \\
 \vQ_\perp\uvek &\triangleq& \vQ\uve - \vQ_\parallel\uvek.
 \eeqano
With this notation, we now present the state-space collapse result for the MW Scheduler.

\begin{proposition}\label{prop:MWS_Qperp_boundedness}
For the given face $\cF\uk$ and $\vlam\uk\in Relint(\cF\uk)$, consider the MW Scheduler under the arrival process  $\{\vA\uve[t]\}_t$ with a mean rate vector $\vlam\uve \in Int(\cR)$ as defined in (\ref{eqn:projection_k}). Then, there exist finite constants $\{N_r\uk\}_{r=1,2,\cdots}$, independent of $\epsilon,$ such that $\bE\left[\|\bvQ_\perp\uvek\|^r\right] \leq N_r\uk,$ for all $\veps >0 ,$ and each $r=1,2,\cdots.$
\end{proposition}

\begin{proof}
Assuming a given $\veps >0,$ the following argument omits the superscript $\uve$ for notational convenience so that $\vQ\uve, \vlam\uve, \vQ\uvek$ are respectively denoted as $\vQ, \vlam, \vQ\uk$. Let us recall the Lyapunov functions in (\ref{eqn:LyapDefs}) to define their counterparts associated with the normal vector $\vc\uk:$
\beqano
 V_\perp\uk(\vQ)\triangleq \|\vQ_\perp\uk\|, \quad W(\vQ)\triangleq \|\vQ\|^2, \quad \textrm{and} \quad W_\parallel\uk(\vQ)\triangleq \|\vQ_\parallel\uk\|^2,
 \label{eqn:MWS:LyapDefs}
\eeqano
The rest of the proof follows the same line of reasoning as in the proof of Proposition~\ref {prop:JSQ_Qperp_boundedness} to bound $\Delta V_\perp\uk(\vQ)$ by utilizing (\ref{eqn:VperpDriftBound}). However, before we start analyzing the mean conditional drift of $\Delta W(\vQ\uk)$ and $\Delta W_\parallel(\vQ\uk),$ we note the important new fact that: since $\vlam\uk$ is in the relative interior of $\cF\uk,$ there exists a small enough $\delta\uk>0$ such that the set
\beqano
\cB_{\delta\uk}\uk \triangleq \cH\uk \cap \{\vr \in \bR_+^L: \|\vr-\vlam\uk\|\leq \delta\uk\},
\eeqano
denoting the set of vectors on the hyperplane $\cH\uk$ that are within $\delta\uk$ distance from $\vlam\uk,$ lies strictly within the face $\cF\uk.$

We are now ready to upper-bound $\bE[\Delta W(\vQ)\> | \> \vQ]\triangleq\bE[\Delta W(\vQ)\> | \> \vQ[t]=\vQ]$ exactly as in the derivation of (\ref{eqn:JSQcollapse:aux}) to obtain:
\beqa
 \bE[\Delta W(\vQ)\> | \> \vQ]
     &\leq& 2 \>\bE\left[ \La \vQ, \vA-\vS \Ra \> | \> \vQ\right] + K_1, \label{eqn:MWScollapse:aux}
\eeqa
where $K_1 \triangleq L \max(A_{max},S_{max})^2.$ Next, we use the definition of $\vlam$ (recall that $\e$ is omitted here) from (\ref{eqn:projection_k}) to expand the first term as:
\beqa
\bE\left[ \La \vQ, \vA-\vS \Ra \> | \> \vQ\right]
    &=& \La \vQ, \vlam\uk-\euk\vc\uk \Ra -\bE\left[\La \vQ, \vS \Ra \> | \> \vQ\right] \nonumber\\
    &=& -\euk\La \vQ, \vc\uk\Ra 
        +\La \vQ, \vlam\uk-\bE\left[\vS\> | \> \vQ\right] \Ra \nonumber\\
    &=& -\euk\|\vQ_\parallel\uk\|
        +\La \vQ, \vlam\uk-\bE\left[\vS\> | \> \vQ\right] \Ra \label{eqn:MWScollapse:aux1}
\eeqa
We note that the definition of the MW Scheduler (see Definition~\ref{def:MWS}) directly implies that
\beqa
\La \vQ, \bE\left[\vS\> | \> \vQ\right] \Ra &=& \max_{\vr\in\cR} \> \La \vQ,\vr \Ra,
    \label{eqn:MWS:MeanBehavior}
\eeqa
where $\cR$ is described in Definition~\ref{def:SchedulingRateRegionR}. Thus, noting that the set $\cB_{\delta\uk}\uk \subset \cR,$ we can upper-bound (\ref{eqn:MWScollapse:aux1}) as follows:
\beqano
\bE\left[ \La \vQ, \vA-\vS \Ra \> | \> \vQ\right]
    &\leq& - \euk\|\vQ_\parallel\uk\|
        +\min_{\vr \in \cB_{\delta\uk}\uk} \La \vQ, \vlam\uk-\vr \Ra \\
    &\stackrel{(a)}{=}& - \euk\|\vQ_\parallel\uk\|
        +\min_{\vr \in \cB_{\delta\uk}\uk} \La \vQ_\perp\uk, \vlam\uk-\vr \Ra \\
    &\stackrel{(b)}{=}& - \euk\|\vQ_\parallel\uk\|
        - \delta\uk \|\vQ_\perp\uk\|
\eeqano
where (a) follows from the facts that $\vQ_\perp\uk = \vQ - \La \vc\uk,\vQ\Ra \> \vc\uk$ by definition, and that the vector $\vlam\uk-\vr$ is perpendicular to the normal vector $\vc\uk$ since both $\vlam\uk$ and $\vr$ are restricted to be on the hyperplane $\cH\uk;$ (b) follows from a key idea in \cite{sto04} where it is noted that the inner product between $\vQ_\perp\uk$ is minimized when $\vr$ is selected on the boundary of $\cB_{\delta\uk}\uk,$ i.e., with $\|\vlam\uk-\vr\|=\delta\uk,$ such that $\vlam\uk-\vr$ points in the opposite direction to $\vQ_\perp\uk.$ Substituting the above result in (\ref{eqn:MWScollapse:aux}) yields our first bound:
\beqa
 \bE[\Delta W(\vQ)\> | \> \vQ]
     &\leq& - 2 \euk\|\vQ_\parallel\uk\|
        - 2 \delta\uk \|\vQ_\perp\uk\| + K_1. \label{eqn:MWS:WdriftBound}
\eeqa

The lower-bounding argument on $\bE[\Delta W_\parallel\uk(\vQ)\> | \> \vQ]\triangleq\bE[\Delta W_\parallel\uk(\vQ)\> | \> \vQ[t]=\vQ]$ follows the exact same steps until (\ref{eqn:JSQcollapse:aux1}) as in the JSQ case, with the minor modification in the final expression (\ref{eqn:JSQcollapse:aux2}) that the conditioning on $\vQ$ remains for the service vector $\vS$ instead of the arrival vector $\vA,$ since the MW Scheduler controls $\vS$ while the JSQ Router controls $\vA.$ This yields
 \beqa
 \bE[\Delta W_\parallel\uk(\vQ)\> | \> \vQ]
 &\geq& 2\La\vc\uk,\vQ\Ra \La\vc\uk,\vlam-\bE\left[\vS\> | \> \vQ\right]\Ra
        - 2\bE\left[\La\vc\uk,\vS \Ra \La\vc\uk,\vU\Ra \> | \> \vQ\right] \nonumber \\
 &\stackrel{(a)}{\geq}& 2\La\vc\uk,\vQ\Ra \La\vc\uk,\vlam-\bE\left[\vS\> | \> \vQ\right]\Ra - K_2\nonumber \\
 &\stackrel{(b)}{=}& -2 \euk\|\vQ_\parallel\uk\|
     + 2 \|\vQ_\parallel\uk\| \left(\La \vc\uk, \vlam\uk \Ra - \La \vc\uk,\bE\left[\vS\> | \> \vQ\right]\Ra\right) - K_2  \nonumber \\
 &\stackrel{(c)}{\geq}& -2 \euk\|\vQ_\parallel\uk\| - K_2, \label{eqn:MWS:W_par_driftBound}
 \eeqa
where step: (a) is true for $K_2 \triangleq 2 L S_{max}^2;$ (b) uses
the definitions of projection $\vlam\uk$; and (c) follows from the
fact that $\La \vc\uk, \vlam\uk\Ra = b\uk$ since $\vlam\uk \in
\cH\uk,$ and $\{\vr \succeq \vZero: \La \vc\uk,\vr \Ra \leq b\uk\}
\supset \cR$ from (\ref{eqn:SchedulingRateRegionR}).

Finally, utilizing the bounds (\ref{eqn:MWS:WdriftBound}) and (\ref{eqn:MWS:W_par_driftBound}) in (\ref{eqn:VperpDriftBound}), and canceling common terms, yields the following bound on the conditional mean drift of $V_\perp\uk(\vQ)$ as:
\beqano
\bE[\Delta V_\perp\uk(\vQ)\>|\>\vQ] &\leq& -\delta\uk + \frac{K_1+K_2}{2\|\vQ_\perp\uk\|},
\eeqano
where $\delta\uk,$ $K_1,$ and $K_2$ are positive constants defined above, all independent of $\veps.$ By verifying Condition (C1) of Lemma~\ref{lem:Hajek}, this result proves the claimed existence of finite constants $\{N_r\uk\}_{r=1,2,\cdots}$ such that $\bE\left[\|\bvQ_\perp\uvek\|^r\right] \leq N_r\uk,$ for all $\veps >0 ,$ and each $r=1,2,\cdots.$
\hfill
\end{proof}

\section{Upper Bounds and Heavy-Traffic Optimality}
\label{sec:UBs}

The purpose of this section is to derive upper-bounds on the steady-state queue-length through simple Lyapunov-based methods that will be shown to be asymptotically tight under heavy-traffic, i.e., when the arrival rates will approach the boundary of the capacity regions. In particular, we use the state-space collapse results established in Section~\ref{sec:Collapse}.

As in the case of the state-space collapse results, here too the analysis is similar for both the JSQ and MW Policies. The derivation of the upper bounds for both cases use the following lemma.

\begin{lemma}\label{lem:UBFirstSteps}
Consider the generic queueing system described in Section~\ref{sec:model} with queues evolving according to (\ref{eqn:Qevolve}), where the arrival $\vA[t]$ and service $\vS[t]$ vectors at time $t$ are allowed to depend on $\vQ[t]$. To indicate this dependence, we will use the notation $\vA(\vQ)$ and $\vS(\vQ)$ to refer to these two processes. Suppose $\{\vQ[t]\}_t$ converges in distribution to a valid random vector $\bvQ,$ with all bounded moments, i.e., $\bE[\|\bvQ\|^r] < \infty$ for each $r=1,2,\cdots.$ Then, for any positive vector $\vc \in \bR_{++}^L,$ we have
\beqa
\bE\left[ \La \vc, \bvQ \Ra \La \vc, \vS(\bvQ)-\vA(\bvQ) \Ra\right]
    &=& \frac{\bE\left[\La \vc, \vA(\bvQ)-\vS(\bvQ)\Ra^2\right]}{2}
        + \frac{\bE\left[\La \vc, \vU(\bvQ)\Ra^2\right]}{2} \label{eqn:UBFirstStep1}\\
    &&   + \bE\left[\La \vc, \bvQ+\vA(\bvQ)-\vS(\bvQ)\Ra \La \vc, \vU(\bvQ)\Ra \right],\label{eqn:UBFirstStep2}
\eeqa
where $\vU(\bvQ)$ is the random vector of unused service (cf. (\ref{eqn:Qevolve})) in steady state.
\end{lemma}

\begin{proof} Recall the definition of $W_\parallel(\vQ)\triangleq \|\vQ_\parallel\|^2$ from (\ref{eqn:LyapDefs}) and the subsequent definition of its one-step drift $\Delta W_\parallel(\vQ).$ We can expand this drift expression as follows:
 \beqano
 \Delta W_\parallel(\vQ) &=& \left[\La \vc, \vQ[t+1]\Ra^2 - \La \vc, \vQ[t]\Ra^2\right] \> \one(\vQ[t]=\vQ)\\
 &=& \La \vc, \vQ+\vA-\vS+\vU\Ra^2 - \La \vc, \vQ\Ra^2 \\
 &=& \La\vc,\vQ+\vA-\vS\Ra^2 + 2\La\vc,\vQ+\vA-\vS\Ra \La\vc,\vU\Ra
        + \La\vc,\vU\Ra^2 - \La\vc,\vQ\Ra^2\\
 &=& 2\La\vc,\vQ\Ra \La\vc,\vA-\vS\Ra + \La\vc,\vA-\vS\Ra^2
    + 2\La\vc,\vQ+\vA-\vS\Ra \La\vc,\vU\Ra\\&& + \La\vc,\vU\Ra^2
 \eeqano
Now, since $\|\vQ_\parallel\|^2 \leq \|\vQ\|^2$ and $\bE[\|\bvQ\|^2]<\infty,$ we clearly have $\bE[W_\parallel(\bvQ)]<\infty.$ Hence, in steady state, $\Delta W_\parallel$ must have a zero mean. Thus setting $\bE[\Delta W_\parallel(\bvQ)]=0$ gives the desired result.
\hfill
\end{proof}

The state-space collapse result is crucial in bounding the expectation in (\ref{eqn:UBFirstStep2}). To provide some intuition, we first note the following useful property concerning the unused service.

\begin{lemma}\label{lem:UnusedService}
For a queueing network evolving according to (\ref{eqn:Qevolve}) and for any given vector $\vc\succeq \vZero$ with $\|\vc\|=1$, the following property always holds for any policy:
\beqa
\La \vc, \vQ[t+1] \Ra \La \vc, \vU[t] \Ra &=& \La -\vQ_\perp[t+1], \vU[t] \Ra, \label{eqn:UnusedService}
\eeqa
where $\vQ_\perp[t+1] = \vQ[t+1] - \La \vc, \vQ[t+1] \Ra \> \vc.$
\end{lemma}

\begin{proof}
We use the notation $\vQ^+ \triangleq \vQ[t+1]$ and omit $[t]$ for the remaining parameters, $\vQ[t],\vA[t],\vS[t],\vU[t].$ Accordingly, we define the projections and perpendiculars of $\vQ^+$ and $\vU$ with respect to $\vc$ as:
\beqano
\vQ^+_\parallel \triangleq \La \vc, \vQ^+ \Ra \> \vc, \qquad  \vQ^+_\perp \triangleq \vQ^+ - \vQ^+_\parallel, \qquad
\vU_\parallel \triangleq \La \vc, \vU \Ra \> \vc, \qquad  \vU_\perp \triangleq \vU - \vU_\parallel.
\eeqano
The following argument uses the fact that $Q^+_l U_l = 0$ for all $l,$ since either $U_l[t]$ or $Q_l[t+1]$ must be zero for each $t.$ Hence, $\vU$ is orthogonal to $\vQ^+.$ We can now prove the claim:
\beqano
\La \vc, \vQ^+ \Ra \La \vc, \vU \Ra
    &=& \|\vQ^+_\parallel\| \|\vU_\parallel\|
    \> \stackrel{(a)}{=} \> \La \vQ^+_\parallel, \vU_\parallel\Ra\\
    &{=}& \La \vQ^+ - \vQ^+_\perp, \vU_\parallel\Ra
    \> \stackrel{(b)}{=} \> \La \vQ^+, \vU-\vU_\perp \Ra\\
    &\stackrel{(c)}{=}& \La \vQ^+_\perp + \vQ^+_\parallel, -\vU_\perp \Ra
    \>\stackrel{(d)}{=}\> \La -\vQ^+_\perp , \vU_\perp \Ra,\\
     &\stackrel{(e)}{=}& \La -\vQ^+_\perp , \vU \Ra,
\eeqano
where step: (a) corresponds to the equality case of the Cauchy-Schwartz inequality since $\vQ^+_\parallel\succeq \vZero$ and $\vU_\parallel\succeq \vZero$ are aligned by definition; (b) follows from the facts that $\vQ^+_\perp \perp \vU_\parallel$ and $\vU_\parallel = \vU-\vU_\perp;$ (c) utilizes the facts that $\vQ^+ \perp \vU$ and that $\vQ^+ = \vQ^+_\perp + \vQ^+_\parallel;$ (d) uses the fact that $\vQ^+_\parallel \perp \vU_\perp;$ and (e) uses the facts that $\vU = \vU_\perp + \vU_\parallel$ and $\vQ^+_\perp \perp \vU_\parallel.$
\hfill
\end{proof}

Now, we present some intuition regarding the usefulness of the state-space collapse results (cf. Propositions~\ref{prop:JSQ_Qperp_boundedness} and \ref{prop:MWS_Qperp_boundedness}). Using (\ref{eqn:UnusedService}) and noting that $\vQ[t+1] = \vQ[t]+\vA[t]-\vS[t]+\vU[t],$ we can re-write (\ref{eqn:UBFirstStep2}) as:
\beqa
(\ref{eqn:UBFirstStep2})
    &=& \bE\left[\La \vc, \bvQ[t+1]\Ra \La \vc, \vU(\bvQ)\Ra \right]
        - \bE\left[\La \vc, \vU(\bvQ)\Ra^2 \right] \nonumber \\
    &\leq& \bE\left[\La \vc, \bvQ[t+1]\Ra \La \vc, \vU(\bvQ)\Ra \right]
    \nonumber \\
    &=& \bE\left[\La -\bvQ_\perp[t+1], \vU(\bvQ)\Ra \right] \nonumber \\
    &\stackrel{(a)}{\leq}& \sqrt{\bE\left[\|\bvQ_\perp[t+1]\|^2\right] \bE\left[\|\vU(\bvQ)\|^2\right]}\nonumber \\
    &\stackrel{(b)}{=}& \sqrt{\bE\left[\|\bvQ_\perp\|^2\right] \bE\left[\|\vU(\bvQ)\|^2\right]}, \label{eqn:UB:CrossTerm}
\eeqa
where step (a) follows from Cauchy-Schwartz inequality and step (b) is true since the distributions of $\vQ[t+1]$ and $\vQ[t]$ are the same in steady state. The last expression reveals the intuition behind why the state-space collapse results are useful: the expression vanishes as $\epsilon \downarrow 0$ since $\bE\left[\|\bvQ_\perp\|^2\right]$ is uniformly bounded, and unused service goes to zero as $\epsilon \downarrow 0.$  Next, we will build on this intuition to prove heavy-traffic optimality of both the JSQ Routing and MW Scheduling Policies.

\subsection{Upper Bounds and Heavy-Traffic Optimality of JSQ Routing}
\label{sec:UB_JSQ}

The following proposition presents a bound on the steady-state total queue-length under JSQ routing, and establishes the first-moment heavy-traffic optimality of JSQ as the network load approaches the capacity of the network.

\begin{proposition}\label{prop:JSQ_UB_firstmoment}
Consider the routing problem under the exogenous arrival process $\{A_\Sigma\e[t]\}_t$ with mean rate $\lambda_\Sigma\e \in Int(\cR)$ satisfying $\epsilon \triangleq \mu_\Sigma-\lambda_\Sigma > 0,$ and with variance $(\sigma_\Sigma\e)^2.$ Then, under the JSQ Routing Policy, the limiting steady-state queue-length vector $\bvQ\e$ satisfies
 \beqa
 \bE\left[\sum_{l=1}^L \bQ_l\e \right] &\leq& \frac{\zeta\e}{2\epsilon} + \bB_1\e,\label{eqn:JSQ_UB_firstmoment}
 \eeqa
where we recall that $\zeta\e \triangleq (\sigma_\Sigma\e)^2 + \nu_\Sigma^2$ as it is defined in Section~\ref{sec:LB_Routing}, and $\bB_1\e$ is
$o\left(\frac{1}{\epsilon}\right),$ i.e., $\dsp\lim_{\epsilon\downarrow 0} \epsilon \bB_1\e = 0.$

Also, in the \textbf{heavy traffic limit}, where we consider a sequence of exogenous arrival processes $\{A_\Sigma\e[t]\}_t$ with $\epsilon \downarrow 0$ so that $\lambda_\Sigma\e$ approaches $\mu_\Sigma$ and $(\sigma_\Sigma\e)^2$ approaches a constant $\sigma_\Sigma^2$, we have
 \beqa
 \limsup_{\epsilon \downarrow 0}
    \epsilon \> \bE\left[\sum_{l=1}^L \bQ_l\e \right]
        &\leq& \frac{\zeta}{2}
            \label{eqn:JSQ_UB_HT_firstmoment},
 \eeqa
where $\zeta \triangleq \sigma_\Sigma^2 + \nu_\Sigma^2.$

Hence, comparing the heavy-traffic lower-bound (\ref{eqn:LB_HT_firstmoment}) for any feasible policy to the heavy-traffic upper-bound (\ref{eqn:JSQ_UB_HT_firstmoment}) for JSQ Router establishes the \emph{first moment heavy-traffic optimality of JSQ Routing Policy}.
\end{proposition}

\begin{proof}
Recalling the definition of $\vc \triangleq \frac{\vOne}{\sqrt{L}}$ in the JSQ case, we first note a useful fact:
\beqa
\bE[\La \vc, \vU(\bvQ\e) \Ra] &=& \frac{\mu_\Sigma-\lambda_\Sigma}{\sqrt{L}} = \frac{\epsilon}{\sqrt{L}}, \label{eqn:JSQ:UbarIsSmall}
\eeqa
which follows from the fact that the mean drift of $\La \vc,\vQ\e \Ra$ must be zero in steady state. Next, we will temporarily omit the superscript $\e$ for ease of exposition and study the terms in (\ref{eqn:UBFirstStep1}) and (\ref{eqn:UBFirstStep2}) under JSQ operation.
\beqa
\bE\left[ \La \vc, \bvQ \Ra \La \vc, \vS(\bvQ)-\vA(\bvQ) \Ra\right]
    &=& \left(\frac{\mu_\Sigma-\lambda_\Sigma}{\sqrt{L}} \right)\bE[\La \vc,\bvQ\Ra]
        \> = \> \frac{\epsilon}{L} \bE\left[ \sum_{l=1}^L \bQ_l \right], \label{eqn:JSQ:UB:aux1}
\eeqa
which follows from the independence of the \emph{total} exogenous arrival process and individual service rate processes from the queue-length levels.
\beqa
\bE\left[\La \vc, \vA(\bvQ)-\vS(\bvQ)\Ra^2\right] &=& \frac{1}{L}\bE\left[\left( A_\Sigma-S_\Sigma \right)^2\right]
    \> = \> \frac{(\sigma_\Sigma^2 + \nu_\Sigma^2 + \epsilon^2)}{L} , \label{eqn:JSQ:UB:aux2}
\eeqa
where we recall that $\sigma_\Sigma^2$ and $\nu_\Sigma^2$ are respectively the variances of the exogenous arrival process $\{A_\Sigma[t]\}_t$ and the hypothetical total service process defined as $S_\Sigma[t] \triangleq \sum_{l=1}^L S_l[t].$
\beqa
\bE\left[\La \vc, \vU(\bvQ)\Ra^2\right]
    &\leq& \La \vc, S_{max} \vOne\Ra \> \bE\left[\La \vc, \vU(\bvQ)\Ra \right]
    \> = \> \frac{\epsilon S_{max}}{L},  \label{eqn:JSQ:UB:aux3}
\eeqa
where we use the bound $U_l \leq S_{max}$ for all $l,$ and the identity (\ref{eqn:JSQ:UbarIsSmall}).

Finally, to bound (\ref{eqn:UBFirstStep2}), we take the same steps as in the argument leading to (\ref{eqn:UB:CrossTerm}), and use the facts that $\vc = \vOne/\sqrt{L}$ and $U_l \leq S_{max}$ for all $l,$ to get:
\beqa
(\ref{eqn:UBFirstStep2})
    &\leq& \sqrt{\bE\left[\|\bvQ_\perp\|^2\right] \bE\left[\|\vU(\bvQ)\|^2\right]}\nonumber \\
    &\leq& \sqrt{\bE\left[\|\bvQ_\perp\|^2\right] S_{max}\sqrt{L}\bE\left[\La \vc, \vU(\bvQ)\Ra\right]}\nonumber \\
    &\leq& \sqrt{\epsilon N_2 S_{max}},    \label{eqn:JSQ:UB:aux4}
\eeqa
where, in the last step, we used (\ref{eqn:JSQ:UbarIsSmall}) and $N_2$ from Proposition~\ref{prop:JSQ_Qperp_boundedness}.

We reintroduce the superscript $\e$ to highlight the dependence on $\epsilon$ and substitute (\ref{eqn:JSQ:UB:aux1})-(\ref{eqn:JSQ:UB:aux4}) in (\ref{eqn:UBFirstStep1})-(\ref{eqn:UBFirstStep2}) to get, after minor algebraic manipulations,
\beqano
\bE\left[ \sum_{l=1}^L \bQ_l\e \right] &\leq&
    \frac{((\sigma_\Sigma\e)^2 + \nu_\Sigma^2 + \epsilon^2)}{2 \epsilon}
    +\frac{S_{max}}{2} + L \sqrt{\frac{N_2 \>  S_{max}}{\epsilon} }
            \> = \> \frac{\zeta\e}{2\epsilon} + \bB_1\e,
\eeqano
where $\zeta\e \triangleq (\sigma_\Sigma\e)^2 + \nu_\Sigma^2 +\epsilon^2$ and $\bB_1\e \triangleq L \sqrt{\frac{N_2  \>S_{max}}{\epsilon} } + \frac{S_{max}}{2},$ which is $o(1/\epsilon)$ as claimed. Then, (\ref{eqn:JSQ_UB_HT_firstmoment}) follows immediately by taking the limit of both sides.
\hfill
\end{proof}

\subsection{Upper Bounds and Heavy-Traffic Optimality of MW Scheduling}
\label{sec:UB_MWS}

The following proposition yields upper bounds on the steady-state weighted total queue-length under MW Scheduling, and then establishes the first-moment heavy-traffic optimality of MWS as the network load approaches the boundary of the capacity region of the network.

\begin{proposition}\label{prop:MWS_UB_firstmoment}
Consider the scheduling problem under the exogenous arrival vector process $\{\vA\uve[t]\}_t$ with mean vector $\vlam\uve \in Int(\cR)$ as defined in (\ref{eqn:projection_k}), and with the variance vector $(\vsig\uve)^2\triangleq \left((\sigma_l\uve)^2\right)_{l=1}^L$. Then, under MW Scheduling, the limiting steady-state queue-vector $\bvQ\uve$ satisfies
 \beqa
 \bE\left[\La \vc\uk, \bvQ\uve \Ra\right] &\leq& \frac{\zeta\uvek}{2\euk} + \bB_1\uvek,\label{eqn:MWS_UB_firstmoment}
 \eeqa
where we recall that $\zeta\uvek \triangleq \La (\vc\uk)^2,(\vsig\uve)^2 \Ra$ is defined in Lemma~\ref{lem:Scheduling_LBs}, and $\bB_1\uvek$
is $o\left(\frac{1}{\euk}\right),$ i.e., $\dsp\lim_{\euk\downarrow 0} \euk\bB_1\uvek = 0.$
Consequently, in the \textbf{heavy traffic limit}, we have
 \beqa
 \limsup_{\euk\downarrow 0}
    \euk\bE\left[\La \vc\uk, \bvQ\uve \Ra\right]
        &\leq& \frac{\zeta\uk}{2}
            \label{eqn:MWS_UB_HT_firstmoment},
 \eeqa
where $\zeta\uk \triangleq \La (\vc\uk)^2,\vsig^2 \Ra.$

Hence, comparing the heavy-traffic lower-bound (\ref{eqn:Scheduling_LB_HT_firstmoment}) for any feasible policy to the heavy-traffic upper-bound (\ref{eqn:MWS_UB_HT_firstmoment}) for MW Scheduler establishes the \emph{first moment heavy-traffic optimality of MW Scheduling Policy}.
\end{proposition}

\begin{proof}
We temporarily omit the superscript $\uve$ associated with the arrival and queue-length processes for ease of exposition. Before we investigate (\ref{eqn:UBFirstStep1}) and (\ref{eqn:UBFirstStep2}) for the MW Scheduler, we make several remarks. We first use the fact that the mean drift of $\La \vc\uk,\vQ \Ra$ must be zero in steady-state to get:
\beqa
\bE[\La \vc\uk, \vU(\bvQ) \Ra] &=& \La \vc\uk, \bE[\vS(\bvQ)]\Ra
        - \La \vc\uk, \vlam \Ra \nonumber\\
&\stackrel{(a)}{=}& \La \vc\uk, \bE[\vS(\bvQ)]\Ra - (b\uk-\euk) \nonumber\\
&\stackrel{(b)}{\leq}&  \euk, \label{eqn:MWS:UbarIsSmall}
\eeqa
where (a) follows from (\ref{eqn:projection_k}) and the fact that $\La \vc\uk, \vlam\uk \Ra = b\uk$ since $\vlam\uk\in\cH\uk;$ and (b) follows from the facts that $\bE[\vS(\bvQ)]$ must be in $\cR$ and that $\La \vc\uk, \vr \Ra \leq b\uk$ for all $\vr\in\cR$ by (\ref{eqn:SchedulingRateRegionR}).

Next, we define
\beqa
\pi\uk &\triangleq& \bP\left(\La \vc\uk, \vS(\bvQ)\Ra = b\uk\right), \label{eqn:MWS:def:pik}
\eeqa
to be the fraction of time that the service rate vector is selected from face $\cF\uk$ by the MW Scheduler in steady-state. Also, we define
\beqa
\gamma\uk &\triangleq& \min\{b\uk-\La \vc\uk, \vr \Ra: \textrm{ for all } \vr\in \cS\setminus \cF\uk\}.
\label{eqn:MWS:def:gammak}
\eeqa
Since the set $\cS$ is discrete and finite, $\gamma\uk$ is a \emph{strictly positive} number with a constant value (independent of $\veps$) associated with the geometry of the capacity region $\cR.$ This constant helps us establish the following useful claim associated with $\pi\uk.$
\begin{clm}\label{claim:pik}
For any $\euk\in (0,\gamma\uk),$ we have
\beqa
(1-\pi\uk) &\leq& \frac{\euk}{\gamma\uk}
\label{eqn:MWS:piBound}
\eeqa
where the upper-bound is $O(\euk)$, i.e., vanishes as $\euk\downarrow 0.$
\end{clm}
\begin{proof}(Claim~\ref{claim:pik})
We start with the observation that
\beqano
\bE[\La \vc\uk, \vS(\bvQ)\Ra] &\geq & \La \vc\uk, \vlam \Ra \> = \> b\uk - \euk
\eeqano
where the inequality follows from the stability of the queueing network (cf. Proposition~\ref{lem:MWS_BddMoments}), and the equality follows from the utilization of (\ref{eqn:projection_k}). We can split the left-hand-side into two parts by using the definition of $\pi\uk$ to write
\beqano
\pi\uk b\uk +
    \bE\left[\La \vc\uk, \vS(\bvQ)\Ra
        \> \one\left(\La \vc\uk, \vS(\bvQ)\Ra\neq b\uk\right)\right]
            &\geq & (b\uk - \euk),
\eeqano
which, when re-arranged, leads to the following lower-bound on the expectation:
\beqa
    \bE\left[\La \vc\uk, \vS(\bvQ)\Ra
        \> \one\left(\La \vc\uk, \vS(\bvQ)\Ra\neq b\uk\right)\right]
            &\geq & b\uk \> (1-\pi\uk) - \euk, \label{eqn:MWS:claim:aux1}
\eeqa
Separately, we can upper-bound the same expectation as
\beqano
\bE\left[\La \vc\uk, \vS(\bvQ)\Ra
        \> \one\left(\La \vc\uk, \vS(\bvQ)\Ra=b\uk\right)\right]
            &\leq& (b\uk - \gamma\uk) \> \bE\left[
        \> \one\left(\La \vc\uk, \vS(\bvQ)\Ra \neq b\uk\right)\right] \\
        &=& (b\uk - \gamma\uk) \> (1-\pi\uk),
\eeqano
where the inequality follows from the definition of $\gamma\uk$ in (\ref{eqn:MWS:def:gammak}), and from the equality from the definition of $\pi\uk$ in (\ref{eqn:MWS:def:pik}). Using this bound together with (\ref{eqn:MWS:claim:aux1}) yields (\ref{eqn:MWS:piBound}).
\hfill
\end{proof}
Claim~\ref{claim:pik} implies the following additional fact:\\
$\bE\left[\left(b\uk-\La \vc\uk, \vS(\bvQ)\Ra\right)^2\right]$
\beqa
    &=&  (1-\pi\uk) \bE\left[ \left(b\uk-\La \vc\uk, \vS(\bvQ)\Ra\right)^2 \> \>|\> \left(\La \vc\uk, \vS(\bvQ)\Ra \neq b\uk\right)\right]\nonumber\\
    &\leq& \frac{\euk}{\gamma\uk}
                \left((b\uk)^2+\La \vc\uk, S_{max} \vOne\Ra^2\right)
                    \label{eqn:MWS:diffSq}\\
    &=& O(\euk) \nonumber
\eeqa
where the inequality follows from (\ref{eqn:MWS:piBound}) and the fact that $S_l \leq S_{max}$ for all $l.$ This result establishes in a certain probabilistic sense that $\La \vc\uk, \vS(\bvQ)\Ra$ is close $b\uk$ if $\euk$ is small. This result confirms the intuition that when $\euk$ is small, i.e., when $\vlam$ is close to the face $\cF\uk,$ the MW Scheduler must mostly select service rates on $\cF\uk$  so that the average service rate vector exceeds the given arrival rate vector componentwise to ensure stability.

Our final remark before studying (\ref{eqn:UBFirstStep1}) and (\ref{eqn:UBFirstStep2}) concerns the geometry of the scheduling capacity region $\cR.$ Since the number of possible rate vectors is finite, the number of faces in the rate region is finite. Therefore, for each face $\cF\uk$ of the region $\cR,$ there exists an angle $\theta\uk \in (0,\pi/2]$ such that
\beqa
\La \vc\uk, \vS(\vQ) \Ra &=& b\uk , \qquad \textrm{for all }  \vQ \mbox{ satisfying } \frac{\|\vQ_\parallel\uk\|}{\|\vQ\|}\geq \cos(\theta\uk), \label{eqn:MWS:def:thetak}
\eeqa
where $\vS(\vQ)$ is the service rate vector selected by the MW Scheduler for the given $\vQ$ as in Definition~\ref{def:MWS}.
Note that $\theta\uk$ identifies a \emph{cone} around the line $\vc\uk$ such that any $\vQ$ in the cone leads to a rate allocation on the face $\cF\uk.$

We are now ready to study each term in (\ref{eqn:UBFirstStep1}) and (\ref{eqn:UBFirstStep2}) with $\vc:=\vc\uk$ to establish (\ref{eqn:MWS_UB_firstmoment}).\\

$\bE\left[ \La \vc\uk, \bvQ \Ra \La \vc\uk, \vS(\bvQ)-\vA(\bvQ) \Ra\right]$
\beqa
    &=& \bE\left[ \|\bvQ_\parallel\| \right] (b\uk - \La \vc\uk, \vlam \Ra)
        - \bE\left[ \|\bvQ_\parallel\| (b\uk - \La \vc\uk, \vS(\bvQ)\Ra) \right] \nonumber\\
    &\stackrel{(a)}{=}& \euk\bE\left[ \|\bvQ_\parallel\| \right]
        - \bE\left[ \|\bvQ\| \cos(\theta_{\bvQ,\bvQ_\parallel\uk}) (b\uk-\La \vc\uk, \vS(\bvQ)\Ra) \right] \nonumber\\
    &\stackrel{(b)}{=}& \euk\bE\left[ \|\bvQ_\parallel\| \right]
        - \bE\left[ \|\bvQ\| \cos(\theta_{\bvQ,\bvQ_\parallel\uk}) \one(\theta_{\bvQ,\bvQ_\parallel\uk}> \theta\uk) (b\uk-\La \vc\uk, \vS(\bvQ)\Ra) \right] \nonumber\\
    &\stackrel{(c)}{=}& \euk\bE\left[ \|\bvQ_\parallel\| \right]
        - \bE\left[ \|\bvQ_\perp\uk\| \cot(\theta_{\bvQ,\bvQ_\parallel\uk}) \one(\theta_{\bvQ,\bvQ_\parallel\uk}> \theta\uk) (b\uk-\La \vc\uk, \vS(\bvQ)\Ra) \right] \nonumber\\
    &\stackrel{(d)}{\geq}& \euk\bE\left[ \|\bvQ_\parallel\| \right]
        - \bE\left[ \|\bvQ_\perp\uk\| \one(\theta_{\bvQ,\bvQ_\parallel\uk}> \theta\uk) (b\uk-\La \vc\uk, \vS(\bvQ)\Ra) \right] \cot(\theta\uk) \nonumber\\
    &\geq& \euk\bE\left[ \|\bvQ_\parallel\| \right]
        - \bE\left[ \|\bvQ_\perp\uk\| (b\uk-\La \vc\uk, \vS(\bvQ)\Ra) \right] \cot(\theta\uk) \nonumber\\
    &\stackrel{(e)}{\geq}& \euk\bE\left[ \|\bvQ_\parallel\| \right]
        - \cot(\theta\uk) \sqrt{\bE\left[ \|\bvQ_\perp\uk\|^2\right] \bE\left[ (b\uk-\La \vc\uk, \vS(\bvQ)\Ra)^2 \right]} \nonumber\\
    &\stackrel{(f)}{\geq}& \euk\bE\left[ \|\bvQ_\parallel\| \right]
        - \cot(\theta\uk) \sqrt{\frac{\euk N_2\uk }{\gamma\uk}
                \left((b\uk)^2+\La \vc\uk, S_{max} \vOne\Ra^2\right),}
\label{eqn:MWS:UB:aux1}
\eeqa
where the step (a) follows from (\ref{eqn:projection_k}) and the definition of the angle $\theta_{\vx,\vy}$ between two vectors $\vx$ and $\vy$ given in (\ref{eqn:RL_defs}); (b) is true from the definition of $\theta\uk;$ (c) is true since $\|\bvQ_\perp\uk\| = \|\bvQ\| \sin(\theta_{\bvQ,\bvQ_\parallel\uk});$ (d) is true since cotangent function is a decreasing nonnegative-valued function in $(0,\pi/2];$ (e) follows from Cauchy-Schwartz Inequality; and (f) follows from Proposition~\ref{prop:MWS_Qperp_boundedness} and (\ref{eqn:MWS:diffSq}). We note that, in (\ref{eqn:MWS:UB:aux1}), the first term is $O(\euk)$ while the second term is $O(\sqrt{\euk}).$\\

Next, we bound the first right-hand-side term in (\ref{eqn:UBFirstStep1}):
$\bE\left[\La \vc\uk, \vA(\bvQ)-\vS(\bvQ)\Ra^2\right]$
\beqa
    &\stackrel{(a)}{=}& \bE\left[(\La \vc\uk, \vA \Ra-b\uk)^2\right]
        + \bE\left[(b\uk-\La \vc\uk, \vS(\bvQ) \Ra)^2\right] \nonumber \\
    & & + 2 \> \left(\La \vc\uk, \vlam \Ra-b\uk\right)\bE\left[b\uk-\La \vc\uk, \vS(\bvQ) \Ra\right] \nonumber\\
    &\stackrel{(b)}{=}& \bE\left[(\La \vc\uk, \vA \Ra-b\uk)^2\right]
        + \bE\left[(b\uk-\La \vc\uk, \vS(\bvQ) \Ra)^2\right] \nonumber
        - 2 \> \euk\bE\left[b\uk-\La \vc\uk, \vS(\bvQ) \Ra\right] \nonumber\\
    &\stackrel{(c)}{\leq}& \bE\left[(\La \vc\uk, \vA-\vlam \Ra + \La\vc\uk,\vlam\Ra-b\uk)^2 \right]
        + \bE\left[(b\uk-\La \vc\uk, \vS(\bvQ) \Ra)^2\right] \nonumber \\
    &\stackrel{(d)}{=}& \bE\left[\La \vc\uk, \vA-\vlam \Ra^2\right]
                    + 2 \euk\La \vc\uk, \bE[\vA]-\vlam \Ra + (\euk)^2
        + \bE\left[(b\uk-\La \vc\uk, \vS(\bvQ) \Ra)^2\right] \nonumber \\
    &\stackrel{(e)}{\leq}& \La (\vc\uk)^2, \vsig^2 \Ra + (\euk)^2
        + \frac{\euk}{\gamma\uk}
                \left((b\uk)^2+\La \vc\uk, S_{max} \vOne\Ra^2\right) \nonumber\\
    &\stackrel{(f)}{=}& \zeta\uvek
        + \frac{\euk}{\gamma\uk}
                \left((b\uk)^2+\La \vc\uk, S_{max} \vOne\Ra^2\right) \label{eqn:MWS:UB:aux2}
\eeqa
where the step (a) follows simply from expanding the square after adding and subtracting $b\uk,$ and noting that the exogenous arrival rate vector $\vA$ has mean $\vlam$ and is independent of the service rate vector $\vS(\bvQ)$; (b) follows from (\ref{eqn:projection_k}); (c) follows from adding and subtracting $\vlam$ in the first expression, and from noting that $b\uk-\La \vc\uk, \vr \Ra \geq 0$ for any $\vr \in \cR$ and hence for $\vS(\bvQ)$; (d) follows, again, from the definition of $\euk$; (e) follows from (\ref{eqn:MWS:diffSq}) and uses the notation $\vsig^2$ for the variance vector for the arrival process; and (f) uses the definition of the parameter $\zeta\uvek$. Notice that all the terms except the first term vanishes in (\ref{eqn:MWS:UB:aux2}) as $\euk\downarrow 0.$

Next, we bound the last term in (\ref{eqn:UBFirstStep1}) using (\ref{eqn:MWS:UbarIsSmall}) and the fact that $U_l\leq S_{max}$ for all $l:$
\beqa
\bE\left[\La \vc\uk, \vU(\bvQ)\Ra^2\right]
    &\leq& \La \vc, S_{max} \vOne\Ra \> \bE\left[\La \vc\uk, \vU(\bvQ)\Ra \right]
    \> \leq \> \euk\La \vc\uk, S_{max} \vOne\Ra,  \label{eqn:MWS:UB:aux3}
\eeqa
which is also vanishing as $\euk\downarrow 0.$

Finally, we consider the term (\ref{eqn:UBFirstStep2}). While the argument essentially follows that of (\ref{eqn:UB:CrossTerm}), we need to pay more attention to the zero entries of $\vc\uk,$ which did not exist in the JSQ case. To that end, we define $\cL\uk_{++} \triangleq \{l\in\{1,\cdots,L\}: c_l\uk > 0\}$ to denote the strictly positive entries of $\vc\uk.$ Then, we focus on only these components by defining the following restricted vectors living in the $|\cL_{++}\uk|$-dimensional real space:
\beqano
\vctil\uk \triangleq (c_l\uk)_{l\in\cL_{++}\uk},
    \quad \vQtil \triangleq (Q_l)_{l\in\cL_{++}\uk},
    \quad \vUtil \triangleq (U_l)_{l\in\cL_{++}\uk}.
\eeqano
For convenience, we will denote $\vQtil[t+1] = (Q_l[t+1])_{l\in\cL_{++}\uk}$ as $\vQtil^+.$ In this reduced space, we further define the projection and perpendicular of $\vQtil^+$ and $\vUtil$ with respect to $\vctil\uk$ as:
\beqano
\vQtil^+_\parallel \triangleq \La \vctil\uk, \vQtil^+\Ra \> \vctil\uk,
    \quad \vQtil^+_\perp \triangleq \vQtil^+ - \vQtil^+_\parallel,
    \quad \vUtil_\parallel \triangleq \La \vctil\uk, \vUtil\Ra \> \vctil\uk,
    \quad \vUtil_\perp \triangleq \vUtil - \vUtil_\parallel.
\eeqano
Now, since $\vctil\uk \succ \vZero$ satisfies $\|\vctil\uk\|=1,$ the statement of Lemma~\ref{lem:UnusedService} applies to $\vctil\uk, \vQtil^+,$ and $\vUtil$ to yield
\beqano
\La \vctil\uk, \vQtil^+\Ra \La \vctil\uk, \vUtil\Ra
    &=& \La -\vQtil^+_\perp, \vUtil \Ra.
\eeqano
This result, together with the fact that $\La \vc\uk, \vQ^+ \Ra \La \vc\uk, \vU \Ra = \La \vctil\uk, \vQtil^+\Ra \La \vctil\uk, \vUtil\Ra,$ where we used the notation $\vQ^+\triangleq \vQ[t+1]$, allows us to bound (\ref{eqn:UBFirstStep2}) as follows (in the following we temporarily use $\bE_{\bvQ}$ to imply that $\vQ$ is distributed as $\bvQ$ in the expectation):
\beqa
(\ref{eqn:UBFirstStep2}) &\leq&
\bE_{\bvQ}\left[\La \vc\uk, \vQ^+ \Ra \La \vc\uk, \vU \Ra\right]\nonumber\\
&=& \bE_{\bvQ}\left[\La \vctil\uk, \vQtil^+\Ra \La \vctil\uk, \vUtil\Ra\right]\nonumber\\
&=& \bE_{\bvQ}\left[\La -\vQtil^+_\perp, \vUtil \Ra\right]\nonumber\\
&\stackrel{(a)}{\leq}& \sqrt{\bE_{\bvQ}\left[\|\vQtil_\perp^+\|^2\right]
\bE_{\bvQ}\left[\|\vUtil\|^2\right]}\nonumber \\
&\stackrel{(b)}{\leq}& \sqrt{\bE_{\bvQ}\left[\|\vQtil_\perp\|^2\right]
\bE_{\bvQ}\left[\|\vUtil\|^2\right]}, \label{eqn:MWS:UB:aux4a}
\eeqa
where (a) follows from Cauchy-Schwartz inequality, and (b) is true since the distribution of $\vQ^+$ is the same as $\vQ$ under steady-state. Next, we bound the expectations in (\ref{eqn:MWS:UB:aux4a}).

The first expectation of (\ref{eqn:MWS:UB:aux4a}) satisfies:
\beqa
\bE_{\bvQ}\left[\|\vQtil_\perp\|^2\right] &\leq& \bE\left[\|\bvQ_\perp\|^2\right] \> \leq \> N_2\uk, \label{eqn:MWS:UB:aux4b}
\eeqa
where the first inequality follows from the fact that $\vQ_\perp$ and $\vQtil_\perp$ are equal for all positions $\cL_{++}\uk,$ while $\vQ_\perp$ potentially contains additional non-zero components, and hence cannot be smaller in magnitude. The second inequality follows from Proposition~\ref{prop:MWS_Qperp_boundedness} that establishes the state-space collapse of MWS.

The second expectation of (\ref{eqn:MWS:UB:aux4a}) satisfies:
\beqa
\bE_{\bvQ}\left[\|\vUtil\|^2\right]
&=& \bE_{\bvQ}\left[\sum_{l\in\cL_{++}\uk}\Util_l^2\right] \> \> \leq \> \> \frac{S_{max}}{c_{min}\uk} \bE\left[\La \vc\uk,\vU(\bvQ) \ \Ra\right], \label{eqn:MWS:UB:aux4c}
\eeqa
where $c_{min}\uk \triangleq \dsp \min_{m\in \cL_{++}\uk} c_m\uk > 0$. Here, the last inequality is true since $c_l\uk \geq c_{min}\uk$ for all $l\in\cL_{++}\uk$, and $\Util_l \leq S_{max}$ for all $l.$

Substituting the bounds (\ref{eqn:MWS:UB:aux4b}) and (\ref{eqn:MWS:UB:aux4c}) back in (\ref{eqn:MWS:UB:aux4a}) together with the fact (\ref{eqn:MWS:UbarIsSmall}) yields:
\beqa
(\ref{eqn:UBFirstStep2}) &\leq& \sqrt{\euk N_2\uk \frac{S_{max}}{c_{min}\uk}} \label{eqn:MWS:UB:aux4}
\eeqa

To complete, we reintroduce the $\uve, \uvek$ superscript to emphasize the dependence on $\veps$ and substitute the derived bounds (\ref{eqn:MWS:UB:aux1}), (\ref{eqn:MWS:UB:aux2}), (\ref{eqn:MWS:UB:aux3}), (\ref{eqn:MWS:UB:aux4}) in (\ref{eqn:UBFirstStep1}) and (\ref{eqn:UBFirstStep2}) to get
 \beqano
 \bE\left[\La \vc\uk, \bvQ\uve \Ra\right] &\leq& \frac{\zeta\uvek}{2\euk} + \bB_1\uvek,
 \eeqano
where
\beqano
\bB_1\uvek &\triangleq& \cot(\theta\uk)
\sqrt{\frac{N_2\uk \left((b\uk)^2+\La \vc\uk, S_{max} \vOne\Ra^2\right)}{\euk\gamma\uk}}
    \\&&+\frac{\left((b\uk)^2+\La \vc\uk, S_{max} \vOne\Ra^2\right)}{2 \gamma\uk}
   + \frac{\La \vc\uk, S_{max} \vOne\Ra}{2}
 + \sqrt{ \frac{N_2\uk \> S_{max}}{\euk\> c_{min}\uk}},
\eeqano
which is $o(1/\euk)$ as claimed. Then, (\ref{eqn:MWS_UB_HT_firstmoment}) immediately follows by taking the limit as $\euk$.
\hfill
\end{proof}

\section{Some Extensions of the Results on the Scheduling Problem}
\label{sec:extensions}

In this section, we obtain bounds on the $n^{th}$ moment of the steady-state queue lengths for the scheduling problem, and we also discuss how to handle channel fading in the derivations of the bounds. Both of these extensions introduce some challenges, but can be essentially addressed by the methodology presented in the previous sections. The $n^{th}$ moment analysis can also be performed for JSQ routing, but is omitted here since it is quite similar to the scheduling case.

\subsection{$n^{th}$ Moment Analysis}
\label{sec:nthMomentAnalysis}

We follow the steps outlined in Section~\ref{sec:MethodologyOutline} to first develop lower bounds on the $n^{th}$ moment of steady-state queue-lengths, and utilize the state-space collapse result of Section~\ref{sec:Collapse} to find corresponding upper bounds. Then, we will establish, as before, the heavy-traffic optimality of these policies by showing that the appropriate dominant terms of the lower and upper bounds match as the arrival rate vector approaches one of the faces of the capacity region $\cR.$

\paragraph{1. Lower Bounds on the $n^{th}$ moment:} We first extend the approach applied in Section~\ref{sec:LBs} to the $n^{th}$ moment of the lower-bounding system in Figure~\ref{fig:LBqueue}.

\begin{lemma} \label{lem:nthMomemnt:LBs}
For the system of Figure~\ref{fig:LBqueue} with a given service process $\{\beta[t]\}_t$, consider the arrival process $\{\alpha\e[t]\}_t,$ parameterized by $\epsilon>0,$ with mean $\alpha\e$ satisfying $\epsilon = \beta-\alpha\e$, and with variance denoted as $\sigma_{\alpha\e}^2$. Let the queue-length process, denoted by $\{\Phi\e[t]\}_t$, evolve as in (\ref{eqn:LB_Qevolve}) with $\alpha[t]:=\alpha\e[t]$.

Then, $\{\Phi\e[t]\}_t$ is a positive Harris recurrent Markov Chain (\cite{meytwe93}) for any $\epsilon>0,$ and converges in distribution to a random variable $\bPhi\e$ with all bounded moments.

Moreover, the $n^{th}$ moment of $\bPhi\e$ can be lower-bounded  as
 \beqa
 \epsilon^n \> \bE\left[\left(\bPhi\e\right)^n\right] &\geq& n! \left(\frac{\zeta\e}{2}\right)^n - {B_n\e}, \qquad n\geq 1, \label{eqn:LB_nthMoment}
 \eeqa
where $B_n\e$ vanishes with $\epsilon\downarrow 0,$ i.e., $\dsp \lim_{\epsilon \downarrow 0} B_n\e=0$

Therefore, in the \textbf{heavy-traffic limit} as the mean arrival rate approaches the mean service rate from below, i.e., as $\epsilon \downarrow 0,$ and assuming the variance $\sigma_{\alpha\e}^2$ converges to a constant $\sigma_\alpha^2$, the lower bounds become
 \beqa
 \liminf_{\epsilon^n \downarrow 0}
    \epsilon^n \>
        \bE\left[\left(\bPhi\e\right)^n\right] &\geq& n! \left(\frac{\zeta}{2}\right)^n, \qquad n\geq 1, \label{eqn:LB_HT_nthMoment},
 \eeqa
 where $\zeta \triangleq \sigma_{\alpha}^2 + \nu_\beta^2.$
\end{lemma}

\begin{proof}
See Appendix~\ref{app:nth moment}.
\hfill
\end{proof}

Lemma~\ref{lem:nthMomemnt:LBs} reveals an interesting fact that, in the heavy-traffic limit, the dominant terms of the $n^{th}$ moment of $\bPhi\e$ only depends on $\zeta,$ which in turn depends only on the variances of the arrival and service processes. This is consistent with Brownian approximations, which utilize central limit theorem to approximate the system behavior using the first two moments of the arrival and service processes.

We can now apply this generic result to the scheduling problem using the same construction as in Section~\ref{sec:LB_Scheduling}.

\begin{lemma}\label{lem:nthMoment:Scheduling_LBs}
For the scheduling problem of Section~\ref{sec:model_Scheduling} with a given set of feasible schedules $\cS,$ consider the exogenous arrival vector process $\{\vA\uve[t]\}_t$ with mean vector $\vlam\uve \in Int(\cR)$ as defined in (\ref{eqn:projection_k}), and with variance vector denoted as $(\vsig\uve)^2\triangleq \left((\sigma_l\uve)^2\right)_{l=1}^L$. Accordingly, let the queue-length process under MW Scheduling with this arrival process be denoted as $\{\vQ\uve[t]\}_t$, evolving as in (\ref{eqn:Qevolve}). Moreover, let $\bvQ\uve$ denote a random vector with the same distribution as the steady-state distribution of $\{\vQ\uve[t]\}_t$ (by Lemma~\ref{lem:MWS_BddMoments}).

Then, for each $k\in\{1,\cdots,K\},$ and with $\zeta\uvek \triangleq \La (\vc\uk)^2,(\vsig\uve)^2\Ra,$
 \beqa
 (\euk)^{n} \bE\left[\La \vc\uk, \bvQ\uve \Ra^n\right] &\geq& n!\left(\frac{\zeta\uvek}{2}\right)^n - {B_n\uvek} , \qquad n\geq 1, \label{eqn:Scheduling_LB_nthMoment}
 \eeqa
where $B_n\uvek$ vanishes as $\euk\downarrow 0.$

Further, consider the \textbf{heavy-traffic limit} $\euk\downarrow 0;$ and suppose that the variance vector $(\vsig\uve)^2$ approaches a constant vector $\vsig^2.$ Then, defining $\zeta\uk \triangleq \La (\vc\uk)^2,\vsig^2\Ra,$ we have
 \beqa
 \liminf_{\euk\downarrow 0}
    (\euk)^n \bE\left[\La \vc\uk, \bvQ\uve \Ra^n\right]
    &\geq& n! \left(\frac{\zeta\uk}{2}\right)^n, \qquad n\geq 1.
        \label{eqn:Scheduling_LB_HT_nthMoment}
 \eeqa

\end{lemma}

\paragraph{2. State-Space Collapse:} The state space collapse result for MW Scheduling provided in Proposition~\ref{prop:MWS_Qperp_boundedness} applies directly. Next, we utilize this result to develop upper bounds on the $n^{th}$ moment of the steady-state queue-lengths.

\paragraph{3. Upper Bounds and $n^{th}$ Moment Heavy-Traffic-Optimality of MW Scheduling:} The main idea behind the analysis is to utilize the state-space collapse result, which implies that the total unused service under the MW Scheduler is small unless \emph{all} queue-lengths are small. Thus, as the system gets heavily loaded, the total unused service vanishes, making the system act similarly to the lower bounding system investigated above. The following result builds on this to establish the $n^{th}$ moment heavy-traffic optimality of MW Scheduler.

\begin{proposition}\label{prop:MWS_UB_nthMoment}
Consider the scheduling problem under the exogenous arrival vector process $\{\vA\uve[t]\}_t$ with mean vector $\vlam\uve \in Int(\cR)$ as defined in (\ref{eqn:projection_k}), and with variance vector $(\vsig\uve)^2\triangleq \left((\sigma_l\uve)^2\right)_{l=1}^L$. Then, under the MW Scheduling, the limiting steady-state queue-vector $\bvQ\uve$ satisfies
 \beqa
 (\euk)^n \bE\left[\La \vc\uk, \bvQ\uve \Ra^n\right] &\leq& n! \left(\frac{\zeta\uvek}{2}\right)^n + \bB_n\uvek, \qquad n\geq 1, \label{eqn:MWS_UB_nthMoment}
 \eeqa
where we recall that $\zeta\uvek \triangleq \La (\vc\uk)^2,(\vsig\uve)^2 \Ra$ as it is defined in Lemma~\ref{lem:Scheduling_LBs}, and $\bB_n\uvek$
is vanishing as $\euk\downarrow 0.$

Also, in the \textbf{heavy traffic limit}, where we consider a sequence of exogenous arrival processes $\{\vA\uvek[t]\}_t$ with their mean vector $\vlam\uve$ approaching the $k^{th}$ face along its normal $\vc\uk,$ and variance vectors $(\vsig\uve)^2$ approaching a constant vector $\vsig^2,$ we have
 \beqa
 \limsup_{\euk\downarrow 0}
    (\euk)^n \bE\left[\La \vc\uk, \bvQ\uve \Ra^n\right]
        &\leq& n!\left(\frac{\zeta\uk}{2}\right)^n, \qquad n\geq 1,
            \label{eqn:MWS_UB_HT_nthMoment}
 \eeqa
where $\zeta\uk \triangleq \La (\vc\uk)^2,\vsig^2 \Ra.$

Hence, comparing the heavy-traffic lower-bound (\ref{eqn:Scheduling_LB_HT_nthMoment}) for any feasible policy to the heavy-traffic upper-bound (\ref{eqn:MWS_UB_HT_nthMoment}) for MW Scheduler establishes the \emph{$n^{th}$-moment heavy-traffic optimality of MW Scheduling Policy}.
\end{proposition}

\begin{proof}
See Appendix~\ref{app:nth moment upper bound}.
\hfill
\end{proof}

We note that the $n^{th}$ moment argument concerns optimality of the norm of the projection $\|\bvQ_\parallel\uk\|$ onto the vector $\vc\uk.$ It is interesting to relate this norm to the norm of the queue-length vector $\|\bvQ\|$ when $n=2$ to show the $2^{nd}$ moment optimality of MW Scheduling in minimizing $\lim_{\euk\downarrow 0} (\euk)^2 \bE[\|\bvQ\|^2].$ This consequence is provided in the following corollary.

\begin{corollary}\label{cor:2ndMomentResult}
Under the same conditions as in Proposition~\ref{prop:MWS_UB_nthMoment}, the MW Scheduler achieves, in the heavy-traffic limit,
 \beqa
 \limsup_{\euk\downarrow 0}
    (\euk)^2 \bE\left[\|\bvQ\|^2\right]
        &\leq& \frac{(\zeta\uk)^2}{2},
            \label{eqn:MWS_UB_HT_2ndMoment}
 \eeqa
where $\zeta\uk \triangleq \La (\vc\uk)^2,\vsig^2 \Ra.$ Furthermore, since the right-hand-side is the smallest heavy-traffic limit achievable by any policy, this establishes $2^{nd}$-moment heavy-traffic-optimality of MW Scheduling.
\end{corollary}
\begin{proof}
The proof simply follows from noting that
\beqano
\|\vQ_\parallel\uk\|^2 \leq \|\vQ\|^2 = \|\vQ_\parallel\uk\|^2 + \|\vQ_\perp\uk\|^2,
\eeqano
from Pythagorean Theorem, and from combining the state-space collapse result $\bE[\|\vQ_\perp\uk\|^2]\leq N_2\uk$ where $N_2\uk$ is independent of $\euk$ with the tight heavy-traffic lower and upper bounds (\ref{eqn:LB_HT_nthMoment}) and (\ref{eqn:MWS_UB_HT_nthMoment}) on $\bE[\|\vQ_\parallel\uk\|^2].$
\hfill
\end{proof}

\subsection{Channel Fading}
\label{sec:SchedulingWithFading}

We consider the same setup as in Section~\ref{sec:model_Scheduling}, depicted in Figure~\ref{fig:MWS_System}, receiving nonnegative-integer-valued vector of arrivals $\{\vA[t]\}_{t\geq 0}$ with $A_l[t] \leq A_{max}, \forall l,t,$ distributed independently over links and also identically over time. However, instead of a \emph{fixed} set of feasible rate vectors $\cS,$ we allow the feasible set to evolve randomly in time over a \emph{finite} state space. In particular, we let $\{J[t]\}_{t\geq 0}$ be an i.i.d. sequence of random variables capturing the global state of the channel states of all links in the network. We assume that $J[t]\in \cJ$ for some set $\cJ$ with finite cardinality, and let $\psi_j \triangleq \bP(J[t]=j).$ Then, each global state $j \in \cJ$ yields a set of feasible rate vectors $\cS\uj$ that can be provided under that state. We assume that $\cS\uj$ has finite cardinality for each $j$ with $S_l \leq S_{max},$ for all $l$ and each $\vS\triangleq (S_l)_l \in\cS\uj.$ For simplicity, we also assume that each feasible rate vector in $\cS\uj$ is composed of non-negative integers, although the arguments hold for any discrete set of choices.



\paragraph{Capacity Region Under Channel Fading:}
With channel fading, the maximum achievable rate region becomes
\beqano
\bcR &\triangleq& \sum_{j\in \cJ} \psi_j \>{Convex \> Hull}\>(\cS\uj)\\
    \> &=& \> Convex\> Hull \> \left\{\sum_{j\in\cJ} \psi_j \>\vs\uj: \vs\uj \in \cS\uj, \textrm{ for each } j\in\cJ\right\}
\eeqano
which is henceforth called the \emph{Fading Capacity Region}. Notice that the set ${Convex \> Hull}\>(\cS\uj)$ for each $j\in\cJ$ simply yields a convex polyhedral set in $\bR_+^L$ as in (\ref{eqn:SchedulingRateRegionR}). Thus, their finite weighted-sum also yields a polyhedral set that can be equivalently described, with a convenient abuse of notation, as:
 \beqa
 \bcR = \{\vr \geq \vZero: \La \bvc\uk, \vr \Ra \leq \bb\uk, \> k=1,\cdots,\bK\}, \label{eqn:FadingSchedulingRateRegionR}
 \eeqa
where $\bK$ denotes the finite (and minimal) number of hyperplanes that fully describe the polyhedron. We refer the reader to Definition~\ref{def:SchedulingRateRegionR} for the notions of the hyperplane $\bcH\uk$, the pair $(\bvc\uk,\bb\uk),$ and the face $\bcF\uk,$ which identically apply to $\bcR$.

\paragraph{Maximum Weight (MW) Scheduler Under Fading:}
In each slot, the purpose of the scheduler is to select a feasible rate vector from the feasible set to achieve stability of the queueing network in the long run. A well-known generalization (introduced in \cite{taseph92}) of the earlier MW Scheduler (cf. Definition~\ref{def:MWS}) to this fading case is the following: given the queue-length vector $\vQ[t]$ and the global channel state $J[t]$ at the beginning of slot $t,$ the service rate vector is selected as
\beqa
\vS[t] &\triangleq& \vS(\vQ[t],J[t]) \>=\> RAND\left\{\argmax_{\vS \in \cS_{J[t]}} \> \La \vQ[t], \vS\Ra\right\}. \label{eqn:def:MWS:Fading}
\eeqa

Note that the rate vectors in $\bcR$ are typically not instantaneously realizable, but only in the mean sense. In particular, \emph{mean} service rate vector over the channel variations provided by the above MW Scheduler conditioned over a queue-length vector $\vQ$ satisfies:
\beqa
\vR(\vQ) &\triangleq& \bE[\vS[t]\>|\>\vQ[t]=\vQ] \> = \> RAND\left\{\argmax_{\vR \in \bcR} \> \La \vQ, \vR\Ra\right\}. \label{eqn:Fading:RofQ}
\eeqa

We next comment on the application of the three steps of our methodology outlined in Section~\ref{sec:MethodologyOutline} to the steady-state performance of this MW Scheduler in the fading scenario.

\paragraph{1. Moment Bounds and Lower Bounds under Fading:}
The statement and proof of Lemma~\ref{lem:MWS_BddMoments} apply once we replace the non-fading capacity region with the above fading capacity region. However, we note that the specific values of $\bvc\uk$ and $\bb\uk$ are different in the non-fading and fading capacity regions, as the latter also incorporates the channel fading distribution $\vpsi\triangleq(\psi_j)_j.$

Next, we construct the lower-bounding system associated with the $k^{th}$ face of $\bcR,$ for any given $k\in\{1,\cdots,\bK\},$ as in the non-fading case, except that the service process is no longer at a constant rate but must incorporate the channel fading distribution $\vpsi$. Recalling that the pair $(\bvc\uk,\bb\uk)$ describes the associated hyperplane, we first define
$$b\ujk \triangleq \max_{\vs \in \>\cS\uj} \La \bvc\uk, \vs \Ra,
    \qquad \textrm{for each } j\in\cJ,$$
which yields the maximum $\bvc\uk$-weighted service rate available in channel state $j.$

We are now ready to describe the governing arrival and service statistics of the lower bounding system (cf. Figure~\ref{fig:LBqueue}): the arrival process $\{\alpha\uk[t]\}_{t\geq 0}$ of the lower bounding system associated with hyperplane $\bcH\uk$ is set to $\alpha\uk[t]=\La \bvc\uk, \vA[t] \Ra,$ while the service process $\{\beta\uk[t]\}_{t\geq 0}$ is distributed as:
\beqano
\bP\left(\beta\uk[t]=b\ujk\right) &=& \psi_j,
\quad \textrm{ for each } j\in\cJ.
\eeqano
We note that $\{b\ujk\}_{j\in\cJ}$ may be identical for different $j,$ in which case their probabilities are aggregated. It is also true that $\bb\uk = \bE[\beta\uk[t]]$ by construction of $\bcR$.

We note that the queue-length process $\{\Phi\uk[t]\}_{t\geq 0}$ driven by the above arrival $\{\alpha\uk[t]\}_{t\geq 0}$ and service $\{\beta\uk[t]\}_{t\geq 0}$ processes as in (\ref{eqn:LB_Qevolve}) is stochastically smaller than $\{\La \bvc\uk, \vQ[t] \Ra\}_{t\geq 0},$ where $\{\vQ[t]\}_{t\geq 0}$ is the queue-length vector process under \emph{any feasible} scheduling strategy. This follows from a coupling argument that utilizes the facts that: the service process of both the lower bounding and the actual queueing systems are governed by the same fading distribution $\vpsi;$ and that $b\ujk,$ by definition, is the largest $\bvc\uk$-weighted service that can be provided to the queueing system under channel state $j.$ Hence, any lower bound on $\{\Phi\uk[t]\}_t$ is also a lower bound on $\{\La \bvc\uk, \vQ[t] \Ra\}_t.$


The setup and definitions of Section~\ref{sec:Collapse_MWS} apply directly to the fading case given that $\cR$ is the fading capacity region defined above rather than the the non-fading capacity region of Definition~\ref{def:SchedulingRateRegionR}. Then, we have the following equivalent of Lemma~\ref{lem:Scheduling_LBs}.

\begin{lemma}\label{lem:LB:fading:Scheduling}
For the above scheduling problem with a given set of channel fading distribution  $\vpsi=(\psi_j)_j$ and their associated set of feasible schedules $\{\cS\uj\}_j,$ consider the exogenous arrival vector process $\{\vA\uve[t]\}_t$ with mean vector $\vlam\uve \in Int(\bcR)$ as defined in (\ref{eqn:projection_k}), and with variance vector denoted as $(\vsig\uve)^2\triangleq \left((\sigma_l\uve)^2\right)_{l=1}^L$. Accordingly, let the queue-length process under MW Scheduling (see (\ref{eqn:def:MWS:Fading})) with this arrival process be denoted as $\{\vQ\uve[t]\}_t$, evolving as in (\ref{eqn:Qevolve}). Moreover, let $\bvQ\uve$ denote a random random vector with the same distribution as the steady-state distribution of $\{\vQ\uve[t]\}_t$.

Then, we have
 \beqa
 \bE\left[\La \bvc\uk, \bvQ\uve \Ra\right] &\geq& \frac{\zeta\uvek}{2\euk}- B_1\uk \label{eqn:fading:Scheduling_LB_firstmoment}
 \eeqa
where $\zeta\uvek \triangleq \bE\left[\left(\La \bvc\uk,\vA\uve \Ra - \beta\uk \right)^2\right]
= \La (\bvc\uk)^2, (\vsig\uve)^2 \Ra + Var(\beta\uk) + (\euk)^2,$ and $B_1\uk \triangleq \frac{S_{max}}{2}.$

Further, consider the \textbf{heavy-traffic limit} $\euk\downarrow 0;$ and suppose that the variance vector $(\vsig\uve)^2$ approaches a constant vector $\vsig^2.$ Then, defining $\zeta\uk \triangleq \La (\bvc\uk)^2, \vsig^2 \Ra + Var(\beta\uk),$ we have
 \beqa
 \lim_{\euk\downarrow 0}
    \euk\bE\left[\La \bvc\uk, \bvQ\uve \Ra\right]
        &\geq& \frac{\zeta\uk}{2}
            \label{eqn:fading:Scheduling_LB_HT_firstmoment}
 \eeqa

\end{lemma}

Comparing these lower bounds with their non-fading counterparts (\ref{eqn:Scheduling_LB_firstmoment}) and (\ref{eqn:Scheduling_LB_HT_firstmoment}), we note that they include $Var(\beta\uk)$ which captures the impact of the fading distribution $\vpsi$ on the steady-state mean queue-length levels.

\paragraph{2. State-Space Collapse of MW Scheduling under Fading:}
The state-space collapse argument under fading follows the same line of argument as in Section~\ref{sec:Collapse}. First, the statement and proof of Lemma~\ref{lem:CollapseFirstSteps} applies without modification since the maximum service rate is uniformly bounded by $S_{max}$ in every channel state. Next, we follow the development of Section~\ref{sec:Collapse_MWS} to consider the performance of a sequence of systems associated with a sequence of arrival processes $\{\vA\uve[t]\}_{t\geq 0}$ parameterized by $\veps> 0$ as defined in (\ref{eqn:projection_k}).

Then, the corresponding queue-length vector process under MW Scheduling $\{\vQ\uve[t]\}_t$ converge in distribution to $\bvQ\uve.$
Also, the definitions of $\vQ_\parallel\uvek,$ and $\vQ_\perp\uvek$ remain the same as in Section~\ref{sec:Collapse_MWS} with $\cR$ representing the fading capacity region (\ref{eqn:FadingSchedulingRateRegionR}). Then, the following state-space collapse result follows.

%

\begin{proposition}\label{prop:fading:MWS_Qperp_boundedness}
Assume $\vlam\uve \in Int(\bcR)$ as defined in (\ref{eqn:projection_k}). Then, under the MW Scheduling Policy, there exist finite constants $\{N_r\uk\}_{r=1,2,\cdots}$ such that $\bE\left[\|\bvQ_\perp\uvek\|^r\right] \leq N_r\uk,$ for all $\veps > 0,$ and each $r=1,2,\cdots.$
\end{proposition}

\begin{proof}[Outline]
The proof of Proposition~\ref{prop:MWS_Qperp_boundedness} directly applies to this statement with $\cR$ defined as in (\ref{eqn:FadingSchedulingRateRegionR}), and by replacing $\bE[\vS \> | \> \vQ]$ with $\vR(\vQ)$ as defined in (\ref{eqn:Fading:RofQ}). Most importantly, with these substitutions, the key property (\ref{eqn:MWS:MeanBehavior}) continues to hold for the MW Scheduler under fading, which allows the rest of the argument to apply without modification.
\hfill
\end{proof}

\paragraph{3. Upper Bounds and Heavy-Traffic-Optimality of MW Scheduling under Fading:} Similarly to the first two steps, the upper bound arguments of Section~\ref{sec:UBs} extend to the fading scenario with minor modifications. First, Lemma~\ref{lem:UBFirstSteps} applies once we replace $\vS(\bvQ)$ with $\vS(\bvQ,J)$ (see (\ref{eqn:def:MWS:Fading})) to capture the channel randomness, and let the expectation be over the channel fading distribution $\vpsi$ as well.
Then, the statement of Proposition~\ref{prop:MWS_UB_firstmoment} applies almost without modification to the fading case, which is repeated here for convenience.

\begin{proposition}\label{prop:fading:MWS_UB_firstmoment}
Consider the scheduling problem under fading with capacity region (\ref{eqn:FadingSchedulingRateRegionR}) serving the exogenous arrival vector process $\{\vA\uve[t]\}_t$ with mean vector $\vlam\uve \in Int(\bcR)$ as defined in (\ref{eqn:projection_k}), and with variance vector $(\vsig\uve)^2\triangleq \left((\sigma_l\uve)^2\right)_{l=1}^L$. Then, under MW Scheduling, the limiting steady-state queue-vector $\bvQ\uve$ satisfies
 \beqa
 \bE\left[\La \bvc\uk, \bvQ\uve \Ra\right] &\leq& \frac{\zeta\uvek}{2\euk} + \bB_1\uvek,\label{eqn:fading:MWS_UB_firstmoment}
 \eeqa
where we recall that $\zeta\uvek \triangleq \La (\bvc\uk)^2, (\vsig\uve)^2 \Ra + Var(\beta\uk) + (\euk)^2$ as it is defined in Lemma~\ref{lem:LB:fading:Scheduling}, and $\bB_1\uvek$
is $o\left(\frac{1}{\euk}\right),$ i.e., $\dsp\lim_{\euk\downarrow 0} \euk\bB_1\uvek = 0.$

Also, in the \textbf{heavy traffic limit}, where we consider a sequence of exogenous arrival processes $\{\vA\uvek[t]\}_t$ with their mean vector $\vlam_\vA\uve$ approaching the $k^{th}$ dominant face along its normal $\bvc\uk,$ and variance vectors $(\vsig_\vA\uve)^2$ approaching a constant vector $\vsig^2,$ we have
 \beqa
 \lim_{\euk\downarrow 0}
    \euk\bE\left[\La \vc\uk, \bvQ\uve \Ra\right]
        &\leq& \frac{\zeta\uk}{2}
            \label{eqn:fading:MWS_UB_HT_firstmoment},
 \eeqa
where $\zeta\uk \triangleq \La (\bvc\uk)^2, (\vsig\uve)^2 \Ra + Var(\beta\uk).$

Hence, comparing the heavy-traffic lower-bound (\ref{eqn:fading:Scheduling_LB_HT_firstmoment}) for any feasible policy to the heavy-traffic upper-bound (\ref{eqn:fading:MWS_UB_HT_firstmoment}) for MW Scheduler establishes the \emph{first moment heavy-traffic optimality of MW Scheduling Policy under fading}.
\end{proposition}

\begin{proof}[Outline]
We point to a few modifications in the proof of Proposition~\ref{prop:MWS_UB_firstmoment} that yields the proof of this statement. The need for these modifications arise from the fact that the MW scheduler does not directly select its allocation from the fading capacity region $\bcR$, but from the instantaneous feasible set of schedules $\{\cS^{J[t]}\}_t$ available at the time. This subtlety can be handled partly by working with conditional expectation $\vR(\vQ)$ defined in (\ref{eqn:Fading:RofQ}) instead of $\vS(\vQ)$, and partly by introducing conditional probabilities in the analysis.

Omitting the details, we point to the similarities and differences from the non-fading case: the derivation of (\ref{eqn:MWS:UbarIsSmall}) follows with $\vS(\bvQ)$ replaced by $\vR(\bvQ);$ the definition of $\pi\uk$ in (\ref{eqn:MWS:def:pik}) is modified to per channel state $j$ as $$\pi\ujk \triangleq \bP\left(\La \vc\uk, \vS(\bvQ,J)\Ra = b\ujk \> | \> J=j\right),$$ for each $j, k;$ the definition of $\gamma\uk$ in (\ref{eqn:MWS:def:gammak}) is unmodified; the statement of Claim~\ref{claim:pik} is true for the above conditional probability $\pi_j\uk,$ such that $(1-\pi_j\uk) = O(\euk)$ for each channel state $j$. This establishes that under heavy-traffic conditions, the MW Scheduler will operate on the dominant face with high probability in every channel state. The rest of the argument follows the non-fading case since the previous fact allows us to approach the fading case as a time-average of non-fading cases, each operating on rate vectors contributing to the dominant face $\cF\uk.$
\hfill
\end{proof}

\section{Conclusions}

The main contribution of the paper is to show that drift conditions in steady-state can be used to obtain bounds on the moments of queue lengths that are tight in heavy traffic. The key new idea here is to derive an appropriate notion of state-space collapse in steady-state which sharpen the bounds obtained using drift conditions. The results presented in this paper apply to the case where the state collapses to a single dimension. An interesting topic for further research is to understand whether the ideas presented here apply more generally.

\section*{Acknowledgements}
We thank the anonymous reviewers for their useful suggestions, and Bin Li and Siva Theja Maguluri for their careful proof-reading of an earlier version of the paper.

Research supported in part by an AFOSR MURI FA 9550-10-1­0573, Army MURIs W911NF-07-1-0287 and W911NF-08-1-0233, AFOSR Grant FA-9550-08-1-0432, DTRA Grant HDTRA1-08-1-0016, and NSF Grants CAREER-CNS-0953515 and CCF-0916664.

\bibliographystyle{plain}
\bibliography{jsq-mw-refs}

\appendix

\section{Proof of Lemma~\ref{lem:JSQ_BddMoments}}\label{sec:JSQ_BddMoments}

The proof follows from the application of Lemma~\ref{lem:Hajek} to the Markov Chain $X[t]:=\vQ[t]$ using the Lyapunov function $Z(X):= V(\vQ)\triangleq \|\vQ\|.$ First, we check that both conditions (C1) and (C2) are satisfied. We start with (C1):
\beqa
\bE\left[\Delta V(\vQ)\> | \> \vQ[t]=\vQ \right] &=& \bE\left[\|\vQ[t+1]\|-\|\vQ[t]\|]\> | \> \vQ[t]=\vQ \right]\nonumber\\
&=& \bE\left[\sqrt{\|\vQ[t+1]\|^2}
            -\sqrt{\|\vQ[t]\|^2}\> | \> \vQ[t]=\vQ \right] \nonumber\\
&\leq& \frac{1}{2\|\vQ\|}
    \bE\left[\|\vQ[t+1]\|^2
        -\|\vQ[t]\|^2\> | \> \vQ[t]=\vQ \right],\label{eqn:JSQ:MomentBounds:aux1}
\eeqa
where the inequality follows from the fact that $f(x)=\sqrt{x}$ is concave for $x\geq 0$ so that $f(y)-f(x)\leq (y-x)f'(x)=\frac{(y-x)}{2\sqrt{x}}$ with $y:=\|\vQ[t+1]\|^2$ and $x:=\|\vQ[t]\|^2.$  Next, we study the difference in (\ref{eqn:JSQ:MomentBounds:aux1}), which is simply the mean drift of the quadratic Lyapunov function $W(\vQ)\triangleq \|\vQ\|^2.$  We shall omit the time reference $[t]$ after the first step for brevity.
\beqa
 \bE[\Delta W(\vQ)\> | \> \vQ]
     &=& \bE\left[\|\vQ[t+1]\|^2 - \|\vQ\|^2 \> | \> \vQ\right]\nonumber\\
     &=& \bE\left[\|\vQ+\vA-\vS+\vU\|^2 - \|\vQ\|^2 \> | \> \vQ\right]\nonumber\\
     &=& \bE\left[\|\vQ+\vA-\vS\|^2 +2 \La \vQ+\vA-\vS, \vU \Ra +\|\vU\|^2 - \|\vQ\|^2 \> | \> \vQ\right]\nonumber\\
     &\stackrel{(a)}{\leq}& \bE\left[\|\vQ+\vA-\vS\|^2 - \|\vQ\|^2 \> | \> \vQ\right]\nonumber\\
     &=& \bE\left[2 \La \vQ, \vA-\vS \Ra + \|\vA-\vS\|^2 \> | \> \vQ\right]\nonumber\\
     &\leq& 2 \>\bE\left[ \La \vQ, \vA-\vS \Ra \> | \> \vQ\right] + K_1, \label{eqn:JSQ:MomentBounds:aux2}
\eeqa
where the inequality (a) follows from the fact that $U_l(Q_l+A_l-S_l)=-U_l^2\leq 0,$ for each $l,$
and $K_1 \triangleq L \max(A_{max},S_{max})^2$ is finite  since both the arrival and service processes are bounded.

Next, we bound the first term in (\ref{eqn:JSQ:MomentBounds:aux2}) by first defining $\epsilon \triangleq \mu_\Sigma-\lambda_\Sigma,$ and then defining a hypothetical arrival rate vector $\vlam = (\lambda_l)_l$ with respect to the given service rate vector $\vmu$ and $\epsilon$ such that $\lambda_l \triangleq \mu_l - \frac{\epsilon}{L},$ for each $l.$ Note that $\sum_{l=1}^L \lambda_l = \mu_\Sigma - \epsilon = \lambda_\Sigma.$
Then, we can massage the first term in (\ref{eqn:JSQ:MomentBounds:aux2}) as
\beqano
\bE\left[ \La \vQ, \vA-\vS \Ra \> | \> \vQ\right]
    &=& \La \vQ, \bE\left[\vA\> | \> \vQ\right]-\vlam\Ra-\La \vQ, \vmu-\vlam \Ra\\
    &\stackrel{(a)}{=}&
        \La \vQ, \bE\left[\vA\> | \> \vQ\right]-\vlam\Ra-\frac{\epsilon}{L}\La \vQ, \vOne \Ra\\
    &\stackrel{(b)}{=}& \bE[A_\Sigma \> | \> \vQ] \> Q_{min} - \La \vQ, \vlam \Ra -\frac{\epsilon}{L} \|\vQ\|_1\qquad \left\{Q_{min} \triangleq \min_{1\leq m\leq L}  Q_m \geq 0\right\}\\
    &=& \lambda_\Sigma \> Q_{min} - \sum_{l=1}^L \lambda_l Q_l -\frac{\epsilon}{L} \|\vQ\|\\
    &=& -\sum_{l=1}^L \lambda_l (Q_l-\> Q_{min})
        -\frac{\epsilon}{L} \|\vQ\|\\
    &\stackrel{(c)}{\leq}& -\frac{\epsilon}{L} \|\vQ\|
\eeqano
where step (a) follows from the definition of $\vlam;$ (b) follows from the definitions of the JSQ policy (see Definition~\ref{def:JSQ}) and the $l_1$ norm $\|\vQ\|_1\triangleq \sum_{l=1}^L |Q_l|;$ (c) is trivially true since $Q_{min} \leq Q_l$ for all $l.$
Using this bound in (\ref{eqn:JSQ:MomentBounds:aux2}) and back in (\ref{eqn:JSQ:MomentBounds:aux1}) yields
\beqano
\bE\left[\Delta V(\vQ)\> | \> \vQ[t]=\vQ \right]
    &\leq& -\frac{\epsilon}{L} + \frac{K_1}{2 \|\vQ\|},
\eeqano
which verifies Condition (C1). Moving on to Condition (C2), we have
\beqa
|\Delta V(\vQ) |
    &=& \left|\|\vQ[t+1]\|-\|\vQ[t]\|\right|\>\one(\vQ[t]=\vQ)
            \nonumber\\
    &=& \|\vQ[t+1]-\vQ[t]\| \>\one(\vQ[t]=\vQ)\nonumber\\
    &\stackrel{(a)}{\leq}& \|\vQ[t+1]-\vQ[t]\|_1 \>\one(\vQ[t]=\vQ)\nonumber\\
    &\leq& L\max_{1\leq l\leq L}|Q_l[t+1]-Q_l[t]| \>\one(\vQ[t]=\vQ)\nonumber\\
    &\stackrel{(b)}{\leq}& 2 {L}\max(A_{max},S_{max})\label{eqn:JSQ:MomentBounds:aux3}
\eeqa
where (a) follows from the fact that $\|\vx\|_1 \geq \|\vx\|$ for any $\vx\in\bR^L;$ and (b) is true since $A_l[t]$ and $S_l[t]$ are bounded by $A_{max}$ and $S_{max},$ respectively, for all $l=1,\cdots,L$ and all $t\geq0.$ This verifies Condition (C2) and completes the proof.

\section{Proof of Lemma~\ref{lem:CollapseFirstSteps}}\label{sec:CollapseFirstSteps}

We first prove (\ref{eqn:VperpDriftBound}):
\beqano
\Delta V_\perp(\vQ) &=& [\|\vQ_\perp[t+1]\|-\|\vQ_\perp[t]\|]\>\one(\vQ[t]=\vQ) \\
&=& \left[\sqrt{\|\vQ_\perp[t+1]\|^2}
            -\sqrt{\|\vQ_\perp[t]\|^2}\right] \>\one(\vQ[t]=\vQ)\\
&\leq& \frac{1}{2\|\vQ_\perp[t]\|}\left[\|\vQ_\perp[t+1]\|^2
        -\|\vQ_\perp[t]\|^2\right]\>\one(\vQ[t]=\vQ)\\
&=& \frac{1}{2\|\vQ_\perp[t]\|}
    \big[\underbrace{\left(\|\vQ[t+1]\|^2-\|\vQ[t]\|^2\right)\>\one(\vQ[t]=\vQ)}_{=\Delta W(\vQ)}
    -\underbrace{\left(\|\vQ_\parallel[t+1]\|^2
        -\|\vQ_\parallel[t]\|^2\>\one(\vQ[t]=\vQ)\right)}_{=\Delta W_\parallel(\vQ)}\big],
\eeqano
where the inequality follows from the fact that $f(x)=\sqrt{x}$ is concave for $x\geq 0$ so that $f(y)-f(x)\leq (y-x)f'(x)=\frac{(y-x)}{2\sqrt{x}}$ with $y:=\|\vQ_\perp[t+1]\|^2$ and $x:=\|\vQ_\perp[t]\|^2.$ Also, the last step follows from Pythagoras Theorem (\ref{eqn:Pythagorean}) with $\vx:=\vQ_\parallel[\cdot]$ and $\vy:=\vQ_\perp[\cdot].$

Next, we prove (\ref{eqn:VperpMaxBound}):
\beqano
|\Delta V_\perp(\vQ) |
    &=& \left|\|\vQ_\perp[t+1]\|-\|\vQ_\perp[t]\|\right|\>\one(\vQ[t]=\vQ) \\
    &\stackrel{(a)}{\leq}& \|\vQ_\perp[t+1]-\vQ_\perp[t]\| \>\one(\vQ[t]=\vQ)\\
    &\stackrel{(b)}{=}& \|\vQ[t+1]-\vQ[t]-\vQ_\parallel[t]+\vQ_\parallel[t+1]\| \>\one(\vQ[t]=\vQ)\\
    &\stackrel{(c)}{\leq}& \left(\|\vQ[t+1]-\vQ[t]\|+\|\vQ_\parallel[t+1]-\vQ_\parallel[t]\|\right) \>\one(\vQ[t]=\vQ)\\
    &\stackrel{(d)}{\leq}& 2 \|\vQ[t+1]-\vQ[t]\| \>\one(\vQ[t]=\vQ)\\
    &\stackrel{(e)}{\leq}& 2 \sqrt{L}\max_{1\leq l\leq L}|Q_l[t+1]-Q_l[t]| \>\one(\vQ[t]=\vQ)\\
    &\stackrel{(f)}{\leq}& 2 \sqrt{L}\max(A_{max},S_{max})
\eeqano
where: (a) follows from the fact that $| \|\vx\|-\|\vy\| | \leq \|\vx-\vy\|$ for each $\vx, \vy \in\bR^L;$ (b) follows from the definition of $\vQ = \vQ_\perp + \vQ_\parallel;$ (c) follows from triangle inequality; (d) follows from the non-expansive nature of the projection onto a convex set once we note that $\vQ_\parallel[\cdot]$ is the projection of $\vQ[\cdot]$ onto the line along $\vc$, which implies that $\|\vQ_\parallel[t+1]-\vQ_\parallel[t]\|\leq \|\vQ[t+1]-\vQ[t]\|;$ (e) trivially follows from the definition of $\|\cdot\|;$ and (f) is true since $A_l[t]$ and $S_l[t]$ are respectively assumed to be bounded by $A_{max}$ and $S_{max}$ for all $l=1,\cdots,L$ and all $t\geq0.$

\section{Proof of Lemma~\ref{lem:nthMomemnt:LBs}}\label{app:nth moment}

We have already argued the weak convergence and boundedness of all moments of the limiting random variable $\bPhi\e$ for each $\epsilon>0$ in the proof of Lemma~\ref{lem:LBs}. Recalling the evolution (\ref{eqn:LB_Qevolve2}), we study the mean drift of the Lyapunov function $W_n(\Phi) \triangleq \|\Phi\|^n.$ In the following, we temporarily omit the time reference $[t]$ and the superscript $\e$ for ease of exposition:
\beqano
\bE[\Delta W_n(\Phi[t]) \> | \> \Phi[t]=\Phi]
    &=& \bE[ (\Phi+\alpha-\beta+\chi)^n-\Phi^n  \> | \> \Phi] \\
    &=& \bE\left[(\Phi+\alpha-\beta)^n-\Phi^n
        + \sum_{i=0}^{n-1} \left( \begin{array}{cc} n\\i\end{array}\right)
                (\Phi +\alpha-\beta)^i \chi^{n-i}\> | \> \Phi \right]\\
    &\stackrel{(a)}{=}& \bE\left[(\Phi+\alpha-\beta)^n-\Phi^n
        + \chi^n \sum_{i=0}^{n-1} \left( \begin{array}{cc} n\\i\end{array}\right)
                (-1)^i \> | \> \Phi \right]\\
    &{=}& \bE\left[(\Phi+\alpha-\beta)^n-\Phi^n
         - (-\chi)^n \> | \> \Phi \right]\\
    &{=}& \bE\left[\sum_{i=0}^{n-1} \left( \begin{array}{cc} n\\i\end{array}\right) \Phi^i (\alpha-\beta)^{n-i} - (-\chi)^n
        \> | \> \Phi \right]\\
    &{=}& \bE\left[\sum_{i=0}^{n-2} \left( \begin{array}{cc} n\\i\end{array}\right) \Phi^i (\alpha-\beta)^{n-i} + n \Phi^{n-1} (\alpha-\beta)- (-\chi)^n
        \> | \> \Phi \right],
 \eeqano
where (a) uses the fact that $\chi = -(\Phi+\alpha-\beta)\one(\Phi+\alpha-\beta<0)$ by definition of $\chi.$
Taking expectations of both sides with respect to the steady-state distribution, i.e., setting $\Phi[t]=\bPhi,$ and noting that $\bE[\Delta W_n(\bPhi)]=0$ yields:
\beqano
0&=& \bE\left[\sum_{i=0}^{n-2} \left( \begin{array}{cc} n\\i\end{array}\right) \bPhi^i (\alpha-\beta)^{n-i} + n \bPhi^{n-1} (\alpha-\beta)- (-\chi)^n \right]
\eeqano
Re-arranging terms, noting the independence of the arrival and service processes from each other and $\Phi,$ and recalling that $\bE[\beta-\alpha]=\epsilon$ by construction allows us to write:
\beqano
n \epsilon \bE[\bPhi^{n-1}] = \sum_{i=0}^{n-2} \left( \begin{array}{cc} n\\i\end{array}\right) \bE[\bPhi^i]\> \bE\left[(\alpha-\beta)^{n-i}\right] - \bE[(-\chi)^n]
\eeqano
We separate the final term in the summation and multiply both sides with $\epsilon^{n-2}/n$ to get:
\beqa
\epsilon^{n-1} \bE[\bPhi^{n-1}] &=& \frac{(n-1)}{2} \epsilon^{(n-2)} \bE[\bPhi^{n-2}] \bE[(\alpha-\beta)^2]\label{eqn:nthMoment:LB:aux1}\\
& & +  \sum_{i=0}^{n-3} \left( \begin{array}{cc} n\\i\end{array}\right) \frac{\epsilon^{(n-2)}\bE[\bPhi^i]}{n}  \> \bE\left[(\alpha-\beta)^{n-i}\right] - \frac{\epsilon^{(n-2)}\bE[(-\chi)^n]}{n}. \label{eqn:nthMoment:LB:aux2}
\eeqa
Next, we investigate terms in (\ref{eqn:nthMoment:LB:aux2}) to show that they vanish with $\epsilon \downarrow 0.$ To that end,
we first note that
\beqa
\left| \frac{\epsilon^{(n-2)}\bE[(-\chi)^n]}{n} \right|
    &\leq& \left(\frac{\epsilon^{(n-2)}S_{max}^{n-1}}{n}\right) \bE[\chi] \> \> = \> \> \left(\frac{\epsilon^{(n-1)}S_{max}^{n-1}}{n}\right),
    \label{eqn:nthMoment:LB:aux3}
\eeqa
where the final equality follows from the fact that $\bE[\chi]=\epsilon$ under steady-state operation. Clearly, the final expression vanishes with $\epsilon\downarrow 0$ for any $n\geq 2.$

Next, we turn the summation in (\ref{eqn:nthMoment:LB:aux2}) to argue inductively that
$\dsp \lim_{\epsilon\downarrow 0} \epsilon^{n-2} \bE[\bPhi^i] = 0$ for all $n\geq 3$ and for each $i=0,1,\cdots,n-3.$ This result, when proven, implies that the whole sum is vanishing with $\epsilon.$ The first step of the induction holds trivially since $\epsilon \bE[\bPhi^0] = \epsilon.$ Suppose the claim is true for some $n\geq3$ and all $i=0,1,\cdots,n-3.$ We would like to confirm it for $(n+1)$ and $i=0,1,\cdots,n-2$ as well. This is straight-forward for all $i=0,1,\cdots,n-1.$ The case when $i=n-2$ requires us to investigate $\epsilon^{n-1} \bE[\bPhi^{n-2}]$ using the expansion (\ref{eqn:nthMoment:LB:aux1})-(\ref{eqn:nthMoment:LB:aux2}):
\beqano
\epsilon^{n-1} \bE[\bPhi^{n-2}] &=&
\sum_{i=0}^{n-2} \left( \begin{array}{cc} n\\i\end{array}\right) \left(\frac{\epsilon^{n-3}\bE[\bPhi^i]}{n}\right)\> \bE\left[(\alpha-\beta)^{n-i}\right] - \frac{\epsilon^{n-2}\bE[(-\chi)^n]}{n}.
\eeqano
The right-hand-side of this expression vanishes with $\epsilon$ since $\epsilon^{n-2}\bE[\bPhi^i]$ vanishes for each $i=0,1,\cdots,n-3$ by the induction hypothesis, and $\epsilon^{n-2}\bE[(-\chi)^n]$ vanishes according to (\ref{eqn:nthMoment:LB:aux3}). We note that the moments of $(\alpha-\beta)$ are bounded since the arrival and service processes are both bounded. This completes the induction proof, and hence establish that $\dsp \lim_{\epsilon\downarrow 0} \epsilon^{n-2} \bE[\bPhi^i] = 0$ for all $n\geq 3$ and for each $i=0,1,\cdots,n-3.$

Returning to (\ref{eqn:nthMoment:LB:aux1})-(\ref{eqn:nthMoment:LB:aux2}) and revoking the $\e$ notation to highlight the dependence on $\epsilon$, we have thus proven that there exist $\{C_n\e\}_{n\geq 1}$ that vanish with $\epsilon$ such that
\beqa
\epsilon^{n} \bE[(\bPhi\e)^{n}] &\geq& \frac{n}{2} \epsilon^{(n-1)} \bE[(\bPhi\e)^{n-1}] \> \bE[(\alpha\e-\beta)^2] - C_n\e \nonumber\\
&=& n \left(\frac{\zeta\e}{2}\right) \epsilon^{(n-1)} \bE[(\bPhi\e)^{n-1}] - C_n\e\nonumber \\
&\geq& n! \left(\frac{\zeta\e}{2}\right)^n - B_n\e,
\label{eqn:nthMoment:LB:aux}
\eeqa
where $\zeta\e \triangleq (\sigma_\Sigma\e)^2 + \nu_\Sigma^2 + \epsilon^2,$ and $B_n\e \triangleq \sum_{k=1}^{n} \frac{n!}{(n-k)!}\left(\frac{\zeta\e}{2}\right)^k C_{n-k}\e$ also vanishes with $\epsilon$, which proves (\ref{eqn:LB_nthMoment}). The heavy-traffic result (\ref{eqn:LB_HT_nthMoment}) then follows immediately by taking the limit of both sides as $\epsilon\downarrow 0.$

\section{Proof of Proposition~\ref{prop:MWS_UB_nthMoment}}\label{app:nth moment upper bound}

We temporarily omit the superscript $\uve$ associated with the arrival and queue-length processes for ease of exposition. We then study the mean drift of the Lyapunov function $W_n\uk(\vQ) \triangleq \La \vc\uk, \vQ \Ra^n$ associate with the $n^{th}$ moment.
\beqa
\bE[\Delta W_n\uk(\vQ[t])\> |\> \vQ[t] = \vQ]
    &=& \bE\left[ \La \vc\uk, \vQ + \vA - \vS + \vU\Ra^n
            - \La \vc\uk, \vQ \Ra^n \> |\> \vQ\right]\nonumber\\
    &=& \bE\left[ \La \vc\uk, \vQ + \vA - \vS \Ra^n
            - \La \vc\uk, \vQ \Ra^n \> |\> \vQ\right]\nonumber\\
    & & + \sum_{i=0}^{n-1}  \left( \begin{array}{cc} n\\i\end{array}\right)
            \bE\left[\La \vc\uk, \vQ + \vA - \vS \Ra^i\La \vc\uk, \vU \Ra^{n-i}\> |\> \vQ\right]\nonumber\\
    &=& \bE\left[ (\La \vc\uk, \vQ + \vA\Ra -b\uk
                    + b\uk - \La \vc\uk,\vS \Ra)^n
            - \La \vc\uk, \vQ \Ra^n \> |\> \vQ\right]\nonumber\\
    & & + \sum_{i=0}^{n-1}  \left( \begin{array}{cc} n\\i\end{array}\right)
            \bE\left[\La \vc\uk, \vQ + \vA - \vS \Ra^i\La \vc\uk, \vU \Ra^{n-i}\> |\> \vQ\right]\nonumber\\
    &=& \bE\left[ (\La \vc\uk, \vQ + \vA\Ra -b\uk)^n
            - \La \vc\uk, \vQ \Ra^n \> |\> \vQ\right]\nonumber\\
    & & + \sum_{i=0}^{n-1}  \left( \begin{array}{cc} n\\i\end{array}\right)
            \bE\left[(\La \vc\uk, \vQ + \vA \Ra - b\uk)^i
                        (b\uk-\La \vc\uk, \vS \Ra)^{n-i}\> |\> \vQ\right]\nonumber\\
    & & + \sum_{i=0}^{n-1}  \left( \begin{array}{cc} n\\i\end{array}\right)
            \bE\left[\La \vc\uk, \vQ + \vA - \vS \Ra^i\La \vc\uk, \vU \Ra^{n-i}\> |\> \vQ\right]\nonumber\\
    &=& n \La \vc\uk,\vQ \Ra^{n-1}
                \bE[\La \vc\uk, \vA\Ra -b\uk \> |\> \vQ]\label{eqn:nthMoment:UB:aux1}\\
    & &    + \sum_{i=0}^{n-2}  \left( \begin{array}{cc} n\\i\end{array}\right)
            \bE\left[\La \vc\uk, \vQ \Ra^i
                        (\La \vc\uk, \vA \Ra - b\uk)^{n-i}\> |\> \vQ\right]\nonumber\\
    & & + \sum_{i=0}^{n-1}  \left( \begin{array}{cc} n\\i\end{array}\right)
            \bE\left[(\La \vc\uk, \vQ + \vA \Ra - b\uk)^i
                        (b\uk-\La \vc\uk, \vS \Ra)^{n-i}\> |\> \vQ\right]\nonumber\\
    & & + \sum_{i=0}^{n-1}  \left( \begin{array}{cc} n\\i\end{array}\right)
            \bE\left[\La \vc\uk, \vQ + \vA - \vS \Ra^i\La \vc\uk, \vU \Ra^{n-i}\> |\> \vQ\right]\nonumber
\eeqa
Note in (\ref{eqn:nthMoment:UB:aux1}) that
$\bE[\La \vc\uk, \vA\Ra -b\uk \> |\> \vQ] = -\euk$ by construction. Thus, taking the expectation of both sides over the steady-state distribution, i.e. setting $\vQ = \bvQ,$ noting $\bE[\Delta W_n\uk(\bvQ)]=0,$ and multiplying both sides with $(\euk)^{n-2}$ yields:
\beqa
n(\euk)^{n-1} \bE[\La \vc\uk, \bvQ \Ra^{n-1}]=
    \sum_{i=0}^{n-2}  \left( \begin{array}{cc} n\\i\end{array}\right)
             (\euk)^{n-2} \bE\left[\La \vc\uk, \bvQ \Ra^i
                        (\La \vc\uk, \vA \Ra - b\uk)^{n-i}\right]\label{eqn:nthMoment:UB:aux2}\\
    + \sum_{i=0}^{n-1}  \left( \begin{array}{cc} n\\i\end{array}\right)
            (\euk)^{n-2} \bE\left[(\La \vc\uk, \bvQ + \vA \Ra - b\uk)^i
                        (b\uk-\La \vc\uk, \vS(\bvQ) \Ra)^{n-i}\right]\label{eqn:nthMoment:UB:aux3}\\
    + \sum_{i=0}^{n-1}  \left( \begin{array}{cc} n\\i\end{array}\right)
            (\euk)^{n-2} \bE\left[\La \vc\uk, \bvQ + \vA - \vS(\bvQ) \Ra^i\La \vc\uk, \vU(\bvQ) \Ra^{n-i}\right]\label{eqn:nthMoment:UB:aux4}
\eeqa
We note the resemblance of this expression to (\ref{eqn:nthMoment:LB:aux1})-(\ref{eqn:nthMoment:LB:aux2}) of the lower-bounding system with $\bPhi$ replaced with $\La \vc\uk,\bvQ \Ra$. Accordingly, the claimed upper bound (\ref{eqn:MWS_UB_nthMoment}) follows from induction as in the argument (\ref{eqn:nthMoment:LB:aux}) once we establish that (\ref{eqn:nthMoment:UB:aux3}) and (\ref{eqn:nthMoment:UB:aux4}) both vanish as $\euk\downarrow 0.$ Thus, we next study the behavior of each of these expressions with $\euk$ for all $n\geq 2.$

We first study the expectation in (\ref{eqn:nthMoment:UB:aux3}) to show that it is of order $\sqrt{\euk}$ and hence $(\ref{eqn:nthMoment:UB:aux3})$ must vanish as $\euk\downarrow 0$ for any $n\geq 2.$ For any $i=0,1,\cdots,n-1,$ we have\\

$\bE\left[(\La \vc\uk, \bvQ + \vA \Ra - b\uk)^i
    (b\uk-\La \vc\uk, \vS(\bvQ) \Ra)^{n-i}\right]$
\beqa
     &=& \sum_{j=0}^i \left( \begin{array}{cc} i\\j\end{array}\right)
            \bE\left[\La \vc\uk, \bvQ \Ra^j (\La \vc\uk, \vA \Ra - b\uk)^{j-i}
                (b\uk-\La \vc\uk, \vS(\bvQ) \Ra)^{n-i}\right]\nonumber \\
     &=& \sum_{j=0}^i \left( \begin{array}{cc} i\\j\end{array}\right)
            \bE\left[\|\bvQ_\parallel\uk\|^j (b\uk-\La \vc\uk, \vS(\bvQ) \Ra)^{n-i}\right]\bE\left[(\La \vc\uk, \vA \Ra - b\uk)^{j-i}\right]\label{eqn:nthMoment:UB:aux5}
\eeqa
where in the last step we used the independence of arrival processes, and used the definition of $\vQ_\parallel\uk$. Following the same argument as in the derivation of (\ref{eqn:MWS:UB:aux1}) to bound the first expectation as:
\beqano
\bE\left[\|\bvQ_\parallel\uk\|^j (b\uk-\La \vc\uk, \vS(\bvQ) \Ra)^{n-i}\right]
&\leq& (\cot(\theta\uk))^j \bE\left[\|\bvQ_\perp\uk\|^j (b\uk-\La \vc\uk, \vS(\bvQ) \Ra)^{n-i}\right] \\
&\leq& (\cot(\theta\uk))^j \sqrt{\bE\left[\|\bvQ_\perp\uk\|^{2j}\right] \bE\left[ (b\uk-\La \vc\uk, \vS(\bvQ) \Ra)^{2(n-i)}\right]}
\eeqano
We know from Proposition~\ref{prop:MWS_Qperp_boundedness} that $\bE\left[\|\bvQ_\perp\uk\|^{2j}\right] \leq N_{2j}\uk$ for some finite $N_{2j}\uk$ independent of $\euk.$ Also, since $i\in\{0,\cdots,n-1\},$ we can show that
\beqano
\bE\left[ (b\uk-\La \vc\uk, \vS(\bvQ) \Ra)^{2(n-i)}\right]
 &\leq& \frac{\euk}{\gamma\uk} \left((b\uk)^{2(n-i)}+\La \vc\uk, S_{max} \vOne\Ra^{2(n-i)}\right) \> = \> O(\euk)
\eeqano
exactly as argued in (\ref{eqn:MWS:diffSq}). Thus, these two bounds together with the fact that $A_l \leq A_{max}$ for all $l$ establishes that $(\ref{eqn:nthMoment:UB:aux5}) = O\left(\sqrt{\euk}\right),$ which, in turn, proves that $(\ref{eqn:nthMoment:UB:aux3}) = O\left((\euk)^{n-\frac{3}{2}}\right),$
i.e., (\ref{eqn:nthMoment:UB:aux3}) vanishes as $\euk\downarrow 0$ for all $n\geq 2.$

Next, we study the expectations in (\ref{eqn:nthMoment:UB:aux4}) to show that they also vanish as $\euk\downarrow 0$ for all $n\geq 2.$ We recall the $\vctil, \vQtil,\vUtil,\vQtil^+,$ and $\bE_\bvQ[\cdot]$ notation introduced for the derivations (\ref{eqn:MWS:UB:aux4a})-(\ref{eqn:MWS:UB:aux4}) and follow the same line of reasoning. In particular, we note that, for each $i=0,1,\cdots,n-1,$ the expectation in (\ref{eqn:nthMoment:UB:aux4}) can be bounded as
\beqa
\bE_\bvQ\left[\La \vc\uk, \vQ + \vA - \vS \Ra^i\La \vc\uk, \vU \Ra^{n-i}\right] &\leq& \bE_\bvQ\left[\La \vc\uk, \vQ^+ \Ra^i\La \vc\uk, \vU \Ra^{n-i}\right]\nonumber \\
&=& \bE_\bvQ\left[\La \vctil\uk, \vQtil^+ \Ra^i\La \vctil\uk, \vUtil \Ra^{n-i}\right]\nonumber \\
&=& \bE_\bvQ\left[\|\vQtil_\parallel^+\|^i \|\vUtil_\parallel\|^{n-i}\right] \label{eqn:nthMoment:UB:aux6}
\eeqa
To bound (\ref{eqn:nthMoment:UB:aux6}), we consider two cases separately: $i\leq n/2,$ and $i> n/2.$ When $i\leq n/2,$ the argument closely follows (\ref{eqn:MWS:UB:aux4a})-(\ref{eqn:MWS:UB:aux4}) to show that $(\ref{eqn:nthMoment:UB:aux6})=O\left(\sqrt{\euk}\right).$ When $i>n/2,$ we multiply and divide (assuming the non-trivial case of $\|\vUtil_\parallel\|\neq 0$) by $\|\vUtil_\parallel\|^{2i-n}$ to get
\beqano
\bE_\bvQ\left[\|\vQtil_\parallel^+\|^i \|\vUtil_\parallel\|^{n-i}\right]
&=&
\bE_\bvQ\left[\frac{\|\vQtil_\parallel^+\|^i \|\vUtil_\parallel\|^{i}}{\|\vUtil_\parallel\|^{2i-n}}\right]\\
&\stackrel{(a)}{=}&
\bE_\bvQ\left[\frac{\La \vQtil_\perp^+,\vUtil \Ra^i }{\|\vUtil_\parallel\|^{2i-n}}\right]\\
&=& \bE_\bvQ\left[\La \vQtil_\perp^+,\frac{\vUtil}{\|\vUtil_\parallel\|^{2-\frac{n}{i}}} \Ra^i \right]\\
&\stackrel{(b)}{\leq}& \sqrt{\bE_\bvQ\left[ \|\vQtil_\perp^+\|^{2i}\right]
    \bE_\bvQ\left[ \|\vUtil\|^{2(n-i)} \right]}
\eeqano
where (a) follows from Lemma~\ref{lem:UnusedService} similarly as in the derivation of (\ref{eqn:MWS:UB:aux4a}); and (b) follows from Cauchy-Schwartz inequality after minor modifications. It is then easy to show, as in (\ref{eqn:MWS:UB:aux4b}) and (\ref{eqn:MWS:UB:aux4c}), that $\bE_\bvQ\left[ \|\vQtil_\perp^+\|^{2i}\right] \leq N_{2i}\uk$ with $N_{2i}\uk$ defined in Proposition~\ref{prop:MWS_Qperp_boundedness} and $\bE_\bvQ\left[ \|\vUtil\|^{2(n-i)} \right] = O(\euk)$ since $i\leq n-1.$ Hence, we have shown that $(\ref{eqn:nthMoment:UB:aux6})=O\left(\sqrt{\euk}\right)$ in this case as well. Substituting this result back in (\ref{eqn:nthMoment:UB:aux4}) proves the claimed result that $(\ref{eqn:nthMoment:UB:aux4}) = O\left((\euk)^{n-\frac{3}{2}}\right),$
i.e., (\ref{eqn:nthMoment:UB:aux4}) vanishes as $\euk\downarrow 0$ for all $n\geq 2.$

As we noted before, the fact that $(\ref{eqn:nthMoment:UB:aux3})+(\ref{eqn:nthMoment:UB:aux4})$ vanishes when $\euk\downarrow 0$ is sufficient (exactly as in the lower-bound argument of (\ref{eqn:nthMoment:LB:aux})) to inductively derive the upper-bound (\ref{eqn:MWS_UB_nthMoment}) via the recursive relationship in (\ref{eqn:nthMoment:UB:aux2}). Consequently, the heavy-traffic result (\ref{eqn:MWS_UB_HT_nthMoment}) follows immediately, completing the proof.

\end{document}